\documentclass[11pt]{article}
\usepackage{amsmath,epsfig,subfigure,amsthm,amsfonts,amsopn,amsbsy,amssymb,latexsym,amsxtra}

\usepackage[letterpaper=true,colorlinks=true,urlcolor=blue,linkcolor=blue,citecolor=blue,pdfstartview=FitH]{hyperref}

\usepackage{pst-plot}
\usepackage{pstricks}

\allowdisplaybreaks[1]
\newcommand{\beq}[1]{\begin{equation} \label{#1}}
\newcommand{\eeq}{\end{equation}}
\newcommand{\bed}{\begin{displaymath}}
\newcommand{\eed}{\end{displaymath}}

\makeatletter \@addtoreset{equation}{section}

\newcommand{\bea}{\begin{displaymath}\begin{array}{rl}}
\newcommand{\eea}{\end{array}\end{displaymath}}
\newcommand{\ad}{&\!\!\!\disp}
\newcommand{\aad}{&\disp}
\newcommand{\barray}{\begin{array}{ll}}
\newcommand{\earray}{\end{array}}

\newcommand{\lbar}{\overline}

\newcommand{\F}{{\mathfrak F}}
\def\op{{\mathcal L}}
\def\cd{(\cdot)}
\def\M{{\mathcal M}}
\def\rr{{\mathbb R}}
\newcommand{\pr}{{\mathbb P}}
\newcommand{\Q}{\mathbb Q}

\newcommand{\ex}{{\mathbb E}}

\newcommand{\e}{\varepsilon}
\newcommand{\vphi}{{\varphi}}

\newcommand{\al}{\alpha}
\newcommand{\ga}{\gamma}
\newcommand{\sg}{\sigma}

\newcommand{\wdt}{\widetilde}
\newcommand{\wdh}{\widehat}
\newcommand{\wrt}{{with respect to }}

\def\({\left(}
\def\){\right)}
\def\l{\left|}
\def\r{\right|}

\newcommand{\nd}{\noindent}
\def\tr{\hbox{tr}}
\newcommand{\disp}{\displaystyle}

\def\one{{\hbox{1{\kern -0.35em}1}}}
\newcommand{\set}[1]{\left\{#1\right\}}
\newcommand{\abs}[1]{\left\vert #1\right\vert}

\renewcommand{\tilde}{\wdt}

\newcommand{\ph}{\varphi}
\newcommand{\Dl}{\Delta}
\newcommand{\lf}{\lfloor}
\newcommand{\rf}{\rfloor}
\newcommand{\sumd}{\sum^{\lf s/\Dl^\e \rf-1}_{l=0}}

\newcommand{\intl}{\int^{t+l\Dl^\e+\Dl^\e}_{t+l\Dl^\e}}

\newcommand{\ett}{{\ex^\e_{t+l\Dl^\e}}}

\newtheorem{thm}{Theorem}[section]
\newtheorem{prop}%[thm]
{Proposition}[section]
\newtheorem{lem}%[thm]
{Lemma}[section]
\newtheorem{cor}%[thm]
{Corollary}[section]

\theoremstyle{definition}
\newtheorem{rem}%[thm]
{Remark}[section]
{Definition}[section]
\newtheorem{exm}
{Example}[section]

\newcommand{\thmref}[1]{Theorem~{\rm \ref{#1}}}
\newcommand{\lemref}[1]{Lemma~{\rm \ref{#1}}}

\begin{document}
\title{Feynman-Kac Formulas for Regime-Switching Jump Diffusions and their Applications%\\[2ex] {\large Manuscript ID GSSR-2013-0089}
}
\author{Chao Zhu,\thanks{Department of Mathematical
Sciences, University of Wisconsin-Milwaukee, Milwaukee, WI 53201,
{\tt zhu@uwm.edu}. 
% Research of this author was supported in part by the National Science Foundation under DMS-1108782 and a grant from the  UWM Research Growth Initiative.
} \and G. Yin,\thanks{Department of Mathematics, Wayne State University, Detroit, Michigan 48202, {\tt gyin@math.wayne.edu}. 
%Research of this author was supported in part by the National Science Foundation under DMS-1207667.
}
\and
Nicholas A. Baran\thanks{Department of Mathematics, Wayne State
University,  Detroit, MI 48202, U.S.A., {\tt av6658@wayne.edu}.
% Research of this author was supported in part by the National Science Foundation under DMS-1207667.
} }

\maketitle

\begin{abstract}
This work develops Feynman-Kac formulas for a class of regime-switching jump
diffusion processes, in which the jump part is driven by a Poisson random measure associated to a  general L\'evy process
and the switching part depends on the jump diffusion processes. Under broad conditions,
the connections of such stochastic processes and the corresponding partial
integro-differential equations are established.  Related
initial,   terminal,  and boundary value
problems are also treated. Moreover, based on weak convergence of probability measures,
it is demonstrated that
 a sequence of random variables related to the regime-switching jump diffusion process
 converges in distribution to the arcsine law.

\bigskip
\nd{\bf Key words.} Feynman-Kac formula, partial integro-partial differential equation, arcsine law.

\bigskip
\nd{\bf Mathematics Subject Classification.} 60J60, 60J75, 47D08.

\end{abstract}

% \newpage

\setlength{\baselineskip}{0.20in}
\section{Introduction}\label{sect-Intro}

The Feynman-Kac formula establishes natural connections between  partial differential equations (PDEs) and stochastic  processes. For instance, a simple version of the Feynman-Kac formula
 \cite[Section V.3]{IkedaW-89} indicates that for any bounded functions $f, g: \rr \mapsto \rr$ and
any bounded solution $u(t,x)$ of the initial value problem
\begin{equation}
\label{eq-heat-initial}
  \frac{\partial }{\partial t}u(t,x) = \frac{1}{2}\frac{\partial^{2}}{\partial x^{2}} u(t,x) - f(x)  u(t,x), \quad u(0,x) =g(x),
\end{equation} there is a stochastic representation \begin{equation}
\label{eq-stoch-repre-bm}
  u(t,x) = \ex\left[ g(x+ W(t)) \exp\(-\int_{0}^{t} f(x+ W(s))ds\)  \right],
\end{equation} where $W$ is a one-dimensional standard Brownian motion with $W(0)=0$ a.s. Conversely, if we
 define $u$ to be the right-hand side of \eqref{eq-stoch-repre-bm},
 then under some mild regularity conditions on the functions $f$ and $g$, we can show that $u$ is a classical solution to \eqref{eq-heat-initial}.
 First, the Feynman-Kac formula offers a method of solving certain PDEs by simulating   paths of the underlying stochastic processes.
  In addition, a class of expectations of random processes can be computed by  solving the related PDEs. For example, the classical Black-Scholes-Merton PDE can help to determine the arbitrage free price for  European call options \cite[Section 5.8]{Karatzas-S}.

Since the early work of Feynman \cite{feynman1948space} and Kac \cite{Kac-49}, the Feynman-Kac formulas have been extended and generalized in different  directions.
The Feynman-Kac formula for general multi-dimensional diffusion processes can be found in, for instance,
\cite[Section 5.7]{Karatzas-S}; see also
\cite{MPRZ} for Feynman-Kac representation formula for variational inequalities, and
 \cite[Section 12.2]{Cont-Tankov} and \cite{nualart2001bsde,rong1997BSDE-jump} for
several versions of Feynman-Kac formulas for jump diffusions.
Numerous  applications have been found; see,
for example,
\cite{cont2005,Ekstrom-Tysk-10,Janson-Tysk-06} (finance),
 \cite{band2011dna,Majumdar} (DNA breathing dynamics,  physics, and computer science),
  and \cite{DeM04} (statistical physics, biology, and engineering problems).
 Using switching diffusion models,
  a recent work \cite{Liu14} incorporated continuous-state dependent switching in optimal stopping with applications to
   perpetual American put options. This effort may be extended with the use of switching jump diffusion models,
which opens up possible considerations of
   the Feynman-Kac representation for related problems.

Applications
%on the other  hand
demand the treatment of regime-switching  diffusions with Poisson type jumps. In many real-world applications, the systems often display discontinuous paths as well as structural changes.
Consider, for instance, asset price modeling, in which the commonly used jump diffusion models
\cite{Cont-Tankov}
do not consider the qualitative changes of the volatility, while the regime-switching
Black-Scholes models \cite{Elliott,Zhang}
unrealistically assume the continuity of the price evolution.
In contrast to the references above,
regime-switching jump diffusion processes
%on the other hand
can naturally capture
% these inherent
the features of jump discontinuity as well as random environment changes of the underlying systems.
The Poisson jumps and more general L\'evy jumps are well-known to incorporate both small and big jumps \cite{APPLEBAUM};
while the regime-switching mechanisms provide the structural changes of the systems \cite{MaoY,YZ-10}.
Thus, regime-switching diffusion  with L\'evy jumps
provides a   uniform and realistic
 platform for modeling in a wide range of applications.
Moreover, as we have seen in  \cite{YZ-10},
adding a switching processes in the modeling is not a simple or trivial extension of the
standard models in the literature.

This work aims to develop
the Feynman-Kac formulas and to
establish  connections between
a class of coupled systems of partial  integro-differential   equations  and regime-switching jump diffusion processes.  We will establish three versions of the Feynman-Kac formula (Theorems \ref{thm-feyman1}, \ref{thm-FK-terminal}, and \ref{thm-fk-bdd-region}), corresponding to initial and terminal value Cauchy problems and boundary value problem, respectively.
To the best of our knowledge,
such results are not available in the literature.
The proofs of these results
make essential use of the generalized It\^{o} formula \eqref{eq:ito} and the optional sampling theorem for martingales \cite[Theorem 1.3.22]{Karatzas-S},
and  require careful analysis in handling the (local) martingale terms;  see for instance the proof of Theorem \ref{thm-feyman1}.
 In particular, in presence of a general  L\'evy measure $\nu$ and the form of our stochastic differential equation (see \eqref{levy-measure} and \eqref{sw-jump-diffusion}), the derivation of the Feynman-Kac formulas is not a trivial extension of the counterparts for diffusion or regime-switching diffusion processes; see Remark \ref{rem32-differences}.
 In this work, we
 provide mild conditions under which the Feynman-Kac formulas for regime-switching jump diffusions are derived  rigorously.

Motivated by Kac's derivation of L\'evy's arcsine law for the occupation time for a one-dimensional Brownian motion using the Feynman-Kac formula as well as a recent result of Khasminskii \cite{Khasm99} on arcsine laws for null-recurrent diffusions, we also derive an arcsine law (Proposition \ref{lim-d-T}) for a class of one-dimensional regime-switching jump diffusion processes, in which the switching component is singularly perturbed with fast switching.
We show that the regime-switching jump diffusion converges weakly to a diffusion process, whose diffusion coefficient is determined by an appropriate average of the diffusion coefficients of the subsystems with respect to the invariant measure of the fast switching component; see  Theorem \ref{ep-lim} for the precise statement.
Moreover, we demonstrate by an example that one in general cannot expect $L_2$ convergence corresponding to the weak convergence result established in Theorem \ref{ep-lim}.
A similar phenomenon was recently observed in \cite{Liu-10-strong}.
Nevertheless, the weak convergence result, together with \cite{Khasm99}, will help us to derive the desired arcsine law.

The rest of the paper is arranged as follows.
Section \ref{sect-formulation} presents the formulation of the problem that we wish to study
together with some preliminary results.
Section \ref{sect-feykac} concentrates on obtaining the Feynman-Kac formulas. Section \ref{sect-cauchy} presents the Feynman-Kac formulas for Cauchy problems while
Section \ref{sec:1st} is devoted to the Feynman-Kac formula for
a class of Dirichlet problems.  An example on option pricing in incomplete market is provided in Section \ref{sect-feykac} to demonstrate the utility of our result.
Section \ref{sec:arc} deals with arcsine laws related to the processes of interests.
First,
based on two-time-scale formulation, we examine a system in which the switching component  is fast varying.
This enables us to obtain a limiting  diffusion process  in the sense of weak convergence, which, in turn, helps us to obtain the desired
arcsine law.
Finally, we conclude the paper with additional remarks in Section \ref{sec:conclu}.

\section{Formulation and Preliminary Results}\label{sect-formulation}
To facilitate the
 presentation, we introduce some notation
 that will be used often in later sections.
   Throughout the paper, we use
   $x'$  to denote the transpose of $x$, and
 $x'y$  or $x\cdot y$ interchangeably to denote the inner product of
   the vectors $x$ and $y$.
     For sufficiently smooth $\phi: \rr^n \to \rr$, $D_{x_i} \phi= \frac{\partial \phi}{\partial x_i}$,
    $D_{x_ix_j} \phi= \frac{\partial^2 \phi}{\partial x_i\partial x_j}$, and we denote by $D_x\phi   =(D_{x_1}\phi, \dots, D_{x_n}\phi)'\in \rr^{n}$ and $D^2_{xx}\phi =(D_{x_ix_j}\phi) \in \rr^{n\times n}$ the gradient and  Hessian of $\phi$, respectively.
    For $k \in \mathbb N$, $C^{k}(\rr^{n})$ is the collection of functions $f: \rr^{n }\mapsto \rr$
    with continuous partial derivatives up to the $k$th order while  $C^{k}_{c} (\rr^{n})$ denotes the space of $C^{k}$ functions with compact support.
 If $B $ is a set, we use $B^o$ and $I_B$ to denote the interior and indicator function of $B$, respectively. Throughout the paper, we adopt the conventions that $\sup \emptyset =-\infty$
 and $\inf \emptyset = + \infty$.

\subsection{Formulation}
Let  $(\Omega, \F, \set{\F_{t}}_{t\ge 0}, \pr)$ be  a filtered probability space   satisfying the usual condition on
which is defined  an $n$-dimensional standard $\F_t$-adapted Brownian motion   $W\cd$.
Let $\set{\psi(t)}$ be an $\F_t$-adapted   L\'evy process with L\'evy measure $\nu(\cdot)$.
Denote  by  $N\cd$
     the corresponding $\F_t$-adapted Poisson random measure defined on $\rr_+ \times \rr^n_0$: $$ N(t,U):= \sum_{0 < s \le t}I_{U}( \Delta \psi_{s} )= \sum_{0< s \le t} I_{U}(\psi(s)- \psi(s-)),$$
     where $t \ge 0$ and $U $ is a Borel subset of $\rr^{n}_{0}=\rr^{n}-\set{0}$.  % with compensator $\tilde N$,
     % and intensity measure $\nu\cd$,
   %   where $\rr_+=[0,\infty)$, $\rr^n_0= \rr^n -\set{0}$.
 The compensator $\tilde N$ of $N$ is given by $$\tilde N(dt,dy):= N(dt,dy) - \nu(dy)dt.$$
  Assume that $W\cd$ and $N\cd$ are independent and that
$\nu\cd$ is a  % $\sigma$-finite measure on $\rr_0^n$.
 L\'evy measure so that
 \begin{equation}
\label{levy-measure} \int_{\rr^n_0} (1 \wedge \abs{y}^2)\nu(dy) < \infty,
\end{equation} where  $a_1\wedge a_2 =\min \{a_1, a_2\}$ for  $a_1,a_2 \in \rr $.

 We consider a stochastic differential equation with regime-switching together with L\'evy-type jumps of the form
\begin{equation}\label{sw-jump-diffusion} \begin{aligned}
 dX(t) = & \, b(X(t ),\al(t))dt + \sigma(X(t ),\al(t ))dW(t)  \\
    & \  + \int_{\rr^n_0} \gamma(X(t-),\al(t-),y)\tilde N(dt,dy),
  \ \ t \ge   0,
  \end{aligned} \end{equation}
  with initial conditions
\begin{equation}\label{swjd-initial}
 X(0)=x_0 \in \rr^n, \ \ \al(0) = \al_0\in \M,  \end{equation}
where $b(\cdot,\cdot) : \rr^n \times \M \mapsto \rr^{n}$,
$\sigma (\cdot,\cdot): \rr^n\times \M \mapsto \rr^{n\times n}$, and $ \gamma (\cdot,\cdot,\cdot): \rr^n \times M \times \rr^n_0
\mapsto  \rr^{n}$ are measurable functions, and
 $\al\cd$ is a switching process  with  a finite state space $\M:=\set{1, \dots, m}$ and generator $Q(x)= (q_{ij}(x))\in \rr^{m\times m}$. That is, $\al\cd$ satisfies
\begin{equation}\label{Q-gen}\pr\set{\al(t+ \delta)=j|X(t)=x,
\al(t)=i,X(s),\al(s),s\le t}=\begin{cases}q_{ij}(x)
\delta + o(\delta),&\hbox{ if }\ j\not= i\\
1+ q_{ii}(x)\delta + o(\delta),  &\hbox{ if }\ j=i,
\end{cases}
 \end{equation} as $\delta \downarrow 0$,  where $q_{ij}(x)\ge 0$ for $i,j\in \M$
with $j\not= i$ and $ q_{ii}(x)=-\sum_{j\not= i}q_{ij}(x)<0$
 for each $i\in \M$.

 The evolution of the discrete component
or the switching process $\al\cd$
can be represented by a stochastic integral with respect to a Poisson random
measure; see, for example, \cite{Skorohod-89}. In fact,  for
$x\in \rr^n$ and $i,j
\in \M$ with $j\not=i$, let $\Delta_{ij}(x)$ be
the consecutive  left-closed, right-open
intervals of
the real line, each having length $q_{ij}(x)$. Define a function
$h: \rr^n \times \M \times \rr \mapsto \rr$ by
\begin{equation}\label{h-def} h(x,i,z)=\sum^m_{j=1} (j-i)
I_{\{z\in\Delta_{ij}(x)\}}.
\end{equation}
Then we may write the switching process \eqref{Q-gen} as
a stochastic integral
\begin{equation}\label{eq:jump-mg}
d\al(t) = \int_\rr h(X(t-),\al(t-),z) {N_1}(dt,dy),\end{equation}
where ${N_{1}}(dt,dy)$ is a Poisson random measure  (corresponding to a random point process $\mathfrak p\cd$) with intensity $dt
\times \lambda(dy)$, and $\lambda\cd$
is the Lebesgue measure on $\rr$.  Denote the compensated Poisson random measure of $N_1\cd$
 by $\tilde N_{1}(dt, dy): = N_{1}(dt, dy)-dt\times  \lambda (dy)$. Throughout this paper, we assume that the L\'evy process $\psi\cd$, the random point process $\mathfrak p\cd$,
 and the Brownian motion
$W\cd$ are independent.

 The following condition will be used as our standing assumption throughout the paper.
 \begin{itemize}
  \item[{\rm(A1)}] Assume   that for some positive constant $\kappa$,   we have
\begin{equation}\label{ito-condition}
 \begin{aligned}
& \abs{b(x,i)- b(y,i)}^2 + \abs{\sigma(x,i)- \sigma(y,i)}^2
 \\ & \qquad  \qquad \qquad
 + \int_{\rr^n_0} \abs{\gamma(x,i,z)- \gamma(y,i,z)}^2\nu(dz) \le \kappa \abs{x-y}^2,\\
&
%\abs{ b(x,i)}^2 + \abs{\sigma(x,i)}^2 +
\int_{\rr^n_0} \abs{\ga(x,i,z)}^2 \nu(dz)
 \le \kappa(1+\abs{x}^2),\\
\end{aligned} \end{equation}
 for all $x,y \in \rr^n$  and $i\in \M=\set{1, \dots, m}$, and that
 \begin{equation}
\label{eq-q-bdd}
\sup \set{q_{ij}(x): x \in \rr^{n}, i\not= j \in \M} \le \kappa < \infty.
\end{equation}
\end{itemize}

Under condition (A1),
in view of \cite[Proposition 2.1]{Xi-Yin-11},  for each initial condition $(x_0,\al_0) \in \rr^n \times \M$, the system represented by \eqref{sw-jump-diffusion} and \eqref{Q-gen} (or equivalently, \eqref{sw-jump-diffusion} and \eqref{eq:jump-mg})
has a unique strong solution  $(X\cd,\al\cd)=(X^{x_0,\al_0}\cd,\al^{x_0,\al_0}\cd)$.

\begin{rem}{\rm We note the following facts.
\begin{itemize}
\item[1.] By the  Lipschitz continuity \eqref{ito-condition}, both $b(\cdot,i)$ and $\sg(\cdot,i)$ grow at most linearly.

\item[2.] Because $Q(x)$ depends on $x$,  $(X\cd,\al\cd)$
is a
state-dependent regime-switching jump diffusion. In particular, as in \eqref{Q-gen}, the evolution of the switching component $\al$ depends on the jump diffusion component $X$.
 %Similarly for $X$,
 Equation \eqref{sw-jump-diffusion} shows that the coefficients $b,\sigma$ and $\gamma$ depend on $\alpha$. % Hence the components $X$ and $\alpha$ are correlated and intertwined.
 The $X$ component alone is not necessarily Markovian, but the two-component process $(X,\alpha)$ is.
Note that the model given in  \eqref{sw-jump-diffusion} and \eqref{Q-gen}  is a substantial generalization of the usual Markovian regime-switching jump diffusion.
Indeed, if $Q(x)=Q$, a constant matrix, then $\alpha\cd$ is a Markov chain independent of the Brownian motion $W\cd$ and the  Poisson random measure $N(\cdot,\cdot)$. The formulation then reduces to the commonly used jump diffusion with Markov switching   in the literature.
Treating the regime-switching diffusion counterpart, as demonstrated in \cite{YZ-10},
compared to the usual Markovian regime-switching diffusion
 considered in \cite{MaoY,Zhang},
 the   state-dependent regime-switching   diffusion
  provides a more realistic formulation by allowing the
  %correlation between
 the dependence of $\al\cd$ on
  $X\cd$;
  %and $\al\cd$.
  see, for example, \cite{Liu14,YinZ13,YZ-10} and the references therein for applications of such  state-dependent regime-switching diffusion processes  in areas such as mathematical finance, risk management,  ecosystem modeling, etc.
  This paper further includes L\'evy-type jumps,
  adding additional versatility to the model and complexity to the problem.

  \item[3.] In this paper, the jump part or the discontinuity of $X$ is given by the stochastic integral   with respect to the compensated Poisson random measure $\tilde N$.  As demonstrated in \cite[Chapter 6]{APPLEBAUM}, we can  extend our results in relatively straightforward manners to situations where the jump part is given by $\int_{|y| < c } \gamma(X(t-),\al(t-),y)\tilde N(dt, dy) + \int_{|y| \ge c } \zeta(X(t-),\al(t-),y) N(dt, dy)$ for some  $c \in (0,\infty]$ and appropriate functions $\gamma$ and $ \zeta$. But for ease of presentation,  we choose not to pursue such an extension in this paper.
   \end{itemize}}\end{rem}

The generator of $(X\cd,\al\cd)$ is defined as follows. Denote
$$\begin{aligned}{\mathcal D}_\op : =\biggl\{g: \rr^{n}\times \M \mapsto \rr, \text{ so that for each }i \in \M, \text{ we have }g(\cdot, i) \in C^2 \text{ and } \\  \int_{\rr^n_{0}}\abs{g(x+ \gamma(x,i,z),i)-g(x,i)-
      D_x g(x,i)\cdot\gamma(x,i,z) } \nu(dz)  < \infty\biggl\}.\end{aligned}$$
For $g \in \mathcal D_\op$, we define
\begin{equation}\label{op-defn}\begin{aligned}
\op g(x,i) =  &\, D_x g(x,i)\cdot b(x,i) + \frac{1}{2}\tr((\sigma\sigma')(x,i)D_x^2 g(x,i)) + \sum_{j\in \M} q_{ij}(x) [g(x,j)-g(x,i)]
\\  & \ + \int_{\rr^n_{0}}[g(x+ \gamma(x,i,z),i)-g(x,i)-
      D_x g(x,i)\cdot\gamma(x,i,z) ] \nu(dz), \ (x,i)\in \rr^{d}\times \M.
\end{aligned}\end{equation}
Because of the L\'evy measure $\nu\cd$, $\op g (x,i)$ may not be well-defined if the function $g(\cdot, i)$ is only assumed to be in class $C^{2} $ for each $i\in \M$; see Proposition \ref{prop-dynkin},    Remark \ref{rem-DL}, and Example \ref{exm-1d-DL} for some sufficient conditions for $f \in \mathcal D_{\op}$.

\subsection{Preliminary Results}
This section is devoted to some preliminary results.
Similar to diffusions,
for every  % each $i \in \M$ and each
$f  \in \mathcal{D_{L}}$,
a result known as generalized It\^o's lemma
(see \cite{Skorohod-89,Xi-Yin-11,Yin-Xi-10}) reads
\begin{equation}\label{eq:ito}   f(X(t),\al(t)) -f( X(0),\al(0))
 = \int^t_0 {\cal L}f(X(s-),\al(s-)) ds
+ M_1^f(t) + M_2^f(t) +M_3^f(t),  \end{equation} where
$\cal L$ is the operator associated with the
process $(X,\al)$ defined in \eqref{op-defn},
 and
\begin{equation}\label{Mi-def}\barray \ad M_1^f(t)  = \int^t_0  D_x
f(X(s-),\al(s-)) \cdot \sigma(X(s-),\al(s-))
d W(s), \\
\ad M_2^f(t)  =  \int_0^t \int_{\rr} \big[ f( X(s-), \al(s-)+
h(X(s-),\al(s-),z))
-f(X(s-),\al(s-))\big] \tilde N_{1}(ds, dz), \\
\ad M_3^f(t) = \int_0^t \int_{\rr^n_0} \!\left[f(X(s-) + \gamma(X(s-),\al(s-),z),\al(s-))- f(X(s-),\al(s-))\right]\tilde N(ds, dz).
\earray\end{equation} It is well known that $M_{1}^{f}$ is a local martingale and,
using similar arguments as in \cite[Lemma 2.4]{Yin-Xi-10}, $M_{2}^{f}$ is a local martingale.
 Moreover, $M_3^f$
is a martingale if $f$ is bounded.
In addition, we have the following proposition.

\begin{prop}\label{prop-dynkin}  Assume that {\em (A1)} holds
and that
the function $\gamma \cd$ satisfies for all $(x,i) \in \rr^n \times \M$,
\begin{equation}
\label{add-gamma-conds}
\abs{\gamma(x,i,z)} \le \rho_1(x)< \infty  \text{ if } \abs{z} \le 1,  % \text{ and } \int_{\abs{z}>1} \abs{\gamma(x,i,z)} \nu(dz) < \infty,
\end{equation}
where $\rho_1(x)>0$ depends only on $x$.
Let $f: \rr^n \times \M \mapsto \rr$ be such that for each $i \in \M$, $f(\cdot, i) \in C^2(\rr^n)$ and that
\begin{equation}
\label{eq-f-growth-con}
\abs{f(x,i)} \le K (1+ \abs{x}^2), \text{ for all } (x,i) \in \rr^n \times \M.
\end{equation}
Then $f \in \mathcal D_\op$.
\end{prop}

\begin{proof} % The proof is divided into two steps.
%To show that $f\in \mathcal D_\op$, it suffices
We need to verify that for all $(x,i)\in \rr^n \times \M$,
\begin{equation}
\label{f-integrable}
\int_{\rr_0^n} \abs{f(x+\gamma(x,i,z),i)-f(x)-D_x f(x,i)\cdot \gamma(x,i,z)} \nu(dz) < \infty.
\end{equation}
To this end, we will treat the cases $\abs{z} \le 1 $ and $|z| > 1$ separately.

Using a Taylor expansion, for $|z| \le 1$,  we have %  [?? r.h.s., should it be $\gamma' D^2 f \gamma$? or use trace??]
$$ f(x+\gamma(x,i,z),i)-f(x,i)-D f(x,i)\cdot \gamma(x,i,z)   =   \frac{1}{2} D^2 f(x+\theta \gamma(x,i,z),i) \gamma(x,i,z)  \cdot\gamma(x,i,z)
$$ where
$\theta \in [0,1]$.
 Equation  \eqref{add-gamma-conds} and the fact that $f(\cdot, i) \in C^2$ imply   that
% \begin{displaymath}
%\abs{x+ \theta \gamma(x,i,z) } \le |x| + \abs{\gamma(x,i,z)} \le |x| + C_1(x) := C_2(x)< \infty,
%\end{displaymath}
%which, together with   leads to
$\abs{D^2 f(x+ \theta\gamma(x,i,z),i)}  \le  \rho_2 (x) < \infty$
for some $\rho_2(x)>0$. Then it follows from (A1)  that
\begin{displaymath}
\begin{aligned}
\int_{|z| \le 1}& \abs{ f(x+\gamma(x,i,z),i)-f(x)-D f(x,i)\cdot \gamma(x,i,z)   } \nu(dz)
 \\ &  \le \int_{|z| \le 1} \rho_2(x) \abs{\gamma(x,i,z)}^2 \nu(dz) =\rho_2(x)   \int_{|z| \le 1} \abs{\gamma(x,i,z)}^2 \nu(dz) < \infty.
\end{aligned}
\end{displaymath}

Next for $|z| > 1$, by the quadratic growth condition in \eqref{eq-f-growth-con},
\bea
\ad \abs{f(x+\gamma(x,i,z),i)-f(x,i)-D f(x,i)\cdot \gamma(x,i,z) } \\
\aad \le \abs{f(x+ \gamma(x,i,z),i)}+ \abs{f(x,i)} + \abs{Df(x,i)} \abs{\gamma(x,i,z)} \\
\aad \le K(1+ \abs{x+ \gamma(x,i,z)}^2)+ \abs{f(x,i)} + \abs{Df(x,i)}^{2} + \abs{\gamma(x,i,z)}^{2} \\
\aad \le K_1 (1+ \abs{x}^2 + \abs{f(x,i)} +  \abs{Df(x,i)}^{2} +  \abs{\gamma(x,i,z)}^2),
\eea where $K_{1}$ is some positive constant.
Observe from  \eqref{levy-measure}  that \begin{equation}
\label{levy-finite-B^{c}}
\nu(\rr_{0}^{n}- \set{z \in \rr_{0}^{n}: |z| \le 1}) = \nu (\set{z: |z| > 1})< \infty.
\end{equation}
Then it follows from %the second equation of \eqref{add-gamma-conds},
 \eqref{levy-finite-B^{c}} and Assumption (A1) that
\begin{displaymath}
\begin{aligned}
& \int_{|z|>1} \abs{f(x+\gamma(x,i,z),i)-f(x,i)-D f(x,i)\cdot \gamma(x,i,z) } \nu(dz) \\
& \  \ \le  K_1    \int_{|z|>1}(1+ \abs{x}^2 + \abs{f(x,i)}+ \abs{Df(x,i)}^{2}+  \abs{\gamma(x,i,z)}^2) \nu(dz) \\
&\ \  = K_{1} \left[ \(1+ \abs{x}^2 + \abs{f(x,i)}+ \abs{Df(x,i)}^{2}\)\nu (\set{z: |z| > 1}) + \int_{|z| > 1} |\gamma(x,i,z)|^{2} \nu(dz) \right]
   < \infty.
\end{aligned}\end{displaymath}
Combining the two cases establishes \eqref{f-integrable}.
\end{proof}

\begin{rem}\label{rem-DL}
Alternatively, one can replace condition \eqref{eq-f-growth-con} by the following conditions:
There exist positive constants $K$, $a$, $ b$, and some $\rho_{3} (x)> 0$ such that
\begin{equation}
\label{funct-gamma-cond}
\begin{array}{ll}
\abs{f(x,i) } \le K (1+ \abs{x}^a) , & \text{ for  all } (x,i) \in \rr^n \times \M, \\
\abs{\gamma(x,i,z)} \le K\rho_{3}(x) (1+ \abs{z}^b),        & \text{ for } \abs{z} \ge 1, \\
\int_{|z|> 1} \abs{z}^{ab} \nu (dz)  < \infty. &
  \end{array}\end{equation}
 Then under (A1),  \eqref{add-gamma-conds}, and \eqref{funct-gamma-cond},
 the assertion of Proposition \ref{prop-dynkin} still hold. The proof is similar to that of Proposition \ref{prop-dynkin} and we shall omit the details here.

%%%%% This is not working. (A1) can not be satisfied with such $\nu$ and $\gamma$. %%%%%%%%%
\begin{exm}\label{exm-1d-DL}
Suppose $n=1$, $\nu(dz)= \abs{z}^{-1-\beta}dz$ for some $\beta\in (0,2)$ and
 $$\gamma(x,i,z)=g_{1}(x,i) \abs{z}^{b_{1}} I_{\{|z|  \le  1 \}} + g_{2}(x,i) |z|^{b_{2}} I_{\{|z|  >  1\}}$$ for some $b_{1}  > \frac{\beta}{2}$ and $b_{2}  <  \frac{\beta}{2}$,
 where  $g_{1}$ and $g_{2}$  are continuous functions satisfying % the growth condition
  $$\begin{aligned} &\abs{g_{1}(x,i)}^{2} + \abs{g_{2}(x,i)}^{2}  \le \kappa (1+ |x|^{2}), \\ &  \abs{ g_{1}(x,i) - g_{1}(y,i) } + \abs{ g_{2}(x,i) - g_{2}(y,i) } \le \kappa |x-y|,\end{aligned}$$
  for some $\kappa > 0$ and $(x,i) \in \rr \times \M$.  Consider the operator $\op$ defined in \eqref{op-defn}, in which for simplicity we assume  that $b  =\sigma \equiv 0$. We  have
   $$\int_{\rr_{0}}  |\gamma(x,i,z)|^{2} \nu(dz)  \le 2 \kappa (1+ |x|^{2}) \left[  \int_{(0,1)} z^{2b_{1} -1 -\beta} dz + \int_{[1, \infty)} z^{2b_{2} -1 -\beta} dz  \right] \le K(1+ |x|^{2}), $$ and similarly $$\int_{\rr_{0}} \abs{\gamma(x,i,z) - \gamma(y,i,z)}^{2}\nu(dz) \le K|x-y|^{2},$$ for some $K= K(\kappa, b_{1}, b_{2}, \beta) >0$.  Thus  assumption (A1) is satisfied.

 Suppose for each $i \in \M$, the function $f(\cdot, i) \in C^{2}$ satisfies the first equation of  \eqref{funct-gamma-cond} for some $a\in (0, \frac{\beta}{b_{2}})$.  Clearly both \eqref{add-gamma-conds} and the second equation of \eqref{funct-gamma-cond} are satisfied.  Moreover, it is easy to show that $\int_{|z|> 1} \abs{z}^{ab_{2}} \nu (dz)  = \int_{|z| > 1} |z|^{ab_{2} - 1 - \beta} dz < \infty$, verifying the third equation of  \eqref{funct-gamma-cond}. Thus it follows that $f \in \mathcal D_{\op}$.
 \end{exm}
 %Then clearly both \eqref{add-gamma-conds} and the second equation of \eqref{funct-gamma-cond} are satisfied.
 %Moreover,  assumption  if $a>0$, $ 0< b < \beta\wedge 1$, and $ab < \beta$, then the other equations
 % of %\eqref{add-gamma-conds} and
 % \eqref{funct-gamma-cond} hold as well.
\end{rem}

\begin{cor}
Under Assumption {\em (A1)}, all functions $f:\rr^{n}\times \M \to \rr$ such that $f(\cdot, i ) \in C_{c}^{2}$ for each $i \in \M$  belong to $\mathcal D_{\op}$ and the Dynkin formula \eqref{Dynkin} holds:
\begin{equation}
\label{Dynkin}
\ex_{x,i} \left[ f(X(\tau), \al(\tau))\right]  = f(x,i) + \ex_{x,i}\biggl[\int_0^\tau \op f(X(s-),\al(s-))ds\biggl],
\end{equation} where $\tau$ is a stopping time with $\pr_{x,i}\set{\tau < \infty}=1$.
\end{cor}

\begin{proof} It is easy to see via the Taylor expansion that $f \in \mathcal D_{\op}$. The Dynkin formula \eqref{Dynkin} then follows from  taking expectations on both sides of \eqref{eq:ito} and the optional sampling theorem (\cite[Theorem 1.3.22]{Karatzas-S}).
\end{proof}

In a similar fashion, we can establish the following corollary.
\begin{cor}
Under Assumption {\em (A1)}, all functions $f:\rr^{n}\times \M \to \rr$ such that $f(\cdot, i ) \in C^{2}$ with bounded partial derivatives up to the second order for each $i \in \M$  belong to $\mathcal D_{\op}$ and the Dynkin formula \eqref{Dynkin} holds.
\end{cor}

We end the section with a brief  discussion on the existence and uniqueness for solution to the system represented by
\eqref{sw-jump-diffusion} and \eqref{Q-gen} (or equivalently, \eqref{sw-jump-diffusion} and \eqref{eq:jump-mg})
when assumption (A1) is only satisfied locally.
The global Lipschitz and linear growth conditions in
assumption (A1)
 for the coefficients of \eqref{sw-jump-diffusion}
can be  restrictive in many applications.
 For instance, the  mean-reverting model, the logistic growth model, and the Lotka-Volterra model  do not satisfy the linear growth condition. Therefore it is vital to
  relax assumption (A1).
 The following result gives a set a sufficient conditions under which  system  represented by \eqref{sw-jump-diffusion} and \eqref{Q-gen} (or equivalently, \eqref{sw-jump-diffusion} and \eqref{eq:jump-mg})
 still has a unique strong solution even if (A1) is violated.

\begin{prop}\label{prop-reg-0}
Suppose that
for  each bounded open ball $B(\rho)$
centered at $0$ with radius $\rho$,
Assumption {\em (A1)}  is satisfied
with $\kappa_{\rho}$ replacing the global constant $\kappa$.
Assume also that there is a function $V(\cdot,\cdot): \rr^n \times \M\mapsto \rr_+$
having continuous partial derivatives \wrt $x$ up to the second order  for each $i\in \M$  and satisfying for some positive constants $K_V$ and $\gamma_0$ that
 \begin{align}
 \label{V-integrable}
\int_{\rr_0^n} \abs{V(x+\gamma(x,i,z),i)-V(x)-D_x V(x,i)\cdot \gamma(x,i,z)} \nu(dz) < \infty,\\
\label{V-lip}
\abs{V(x,i)-V(y,i) } \le K_V \abs{x-y}, \text{ for all } x,y \in \rr^n \text{ and } i \in \M, \\
\label{eq:reg-cond}
 \op V(x,i) \le \gamma_0 V(x,i) \ \hbox{ for all } \ (x,i) \in \rr^n\times \M,
\end{align}
and
\begin{align}\label{eq2:reg-cond}
& V_R := \inf_{|x|\ge R, i\in \M} V(x,i) \to \infty \ \hbox{ as } \ R\to \infty.\end{align}
Then  the system represented by \eqref{sw-jump-diffusion} and \eqref{Q-gen} $($or equivalently, \eqref{sw-jump-diffusion} and \eqref{eq:jump-mg}$)$
has a unique strong global solution.
\end{prop}

\begin{proof}
Note that \eqref{V-integrable} guarantees that $V\in  \mathcal D_\op$.
With the given initial condition $(x_{0},\alpha_{0}) \in \rr^{n }\times \M$ as in \eqref{swjd-initial},
since Assumption (A1) is satisfied locally, for any $k\in \mathbb N$ with $|x_{0}| < k$, the system
given by \eqref{sw-jump-diffusion} and \eqref{Q-gen} (or equivalently, \eqref{sw-jump-diffusion} and \eqref{eq:jump-mg})
has a unique strong solution $(X^{(k)}, \al^{(k)}):= (X^{(k),x_{0},\al_{0}}, \al^{(k), x_{0},\al_{0}})$ locally up to the exit time $t\le   \tau_{k}$, where   \begin{equation}
\label{eq-tau-k}
 \tau_k:=\inf\set{t\ge 0: \big\vert X^{(k)}(t)\big\vert \ge k}.
  \end{equation} Moreover, as in \cite[Section 3.4]{K80}, we can construct a sequence $(X^{(k)}\cd,\al^{(k)}\cd)$ in such a way so that $(X^{(\iota)}\cd,\al^{(\iota)}\cd)$ are identical before exiting the ball $B(k)$   for all $ \iota\ge k$.  In particular, we have
  $$ \inf\set{t\ge 0: \big\vert X^{(\iota)}(t) \big\vert\ge k} = \tau_{k}, \text{ for all } \iota \ge k,$$ where $\tau_{k}$ is defined in \eqref{eq-tau-k}.
  Therefore we can define a process $(X,\al)$ so that $(X(t),\al(t)):= (X^{(k)}(t),\al^{(k)}(t))$ for $t < \tau_{k}$.
Clearly $\tau_k$ is an increasing sequence. Denote $\tau_\infty:= \lim_{k\to \infty} \tau_k$. We need to show that $\tau_{\infty}= \infty$ a.s.
 Suppose on the
contrary that the statement were false. Then  there would exist some
 $T>0$ and $\e>0$ such that $\pr_{x,i}\set{\tau_\infty \le T } > \e.$
 Therefore we could find some $k_1 \in \mathbb N$ such that
 \beq{2-beta-k} \pr_{x,i}\set{\tau_k \le T}> \e, \hbox{ for all } k\ge k_1. \eeq
 Define
$$U(x,i,t) = V(x,i) \exp(-\ga_0 t), \ \ (x,i) \in \rr^n \times \M, \hbox{ and }
t\ge 0.$$
Then it satisfies $ [(\partial / \partial t) +\op] U(x,i,t) \le 0$.
Using the generalized It\^o formula \eqref{eq:ito}, we have
\begin{equation}\label{V-Ito}
\begin{aligned}
& V(X(\tau_k \wedge T),\al(\tau_k  \wedge T))\exp(-\ga_0(\tau_k \wedge T )) -V(x,i) \\
 & \  \ =\int_0^{\tau_k \wedge T} e^{-\gamma_0 s} (\op-\gamma_0  ) V(X(s-),\al(s-)) ds + M_1^V(\tau_k\wedge T) + M_2^V(\tau_k\wedge T) + M_3^V(\tau_k\wedge T),
\end{aligned}\end{equation}
where
  \begin{align*} & M_1^V(t)  = \int^{t}_0  e^{-\gamma_0 s}D_x
V(X(s-),\al(s-)) \cdot \sigma(X(s-),\al(s-))
d W(s), \\
 & M_2^V(t)  =  \int^{t}_0\!\!\int_{\rr} e^{-\gamma_0 s} \big[ V( X(s-), \al(s-)+
h(X(s-),\al(s-),z))
-V(X(s-),\al(s-))\big] \tilde N_{1}(ds, dz), \\
&  M_3^V(t) = \int^{t}_0 \!\! \int_{\rr^n_0} e^{-\gamma_0 s} [ V(X(s-) + \gamma(X(s-),\al(s-),z),\al(s-))\\
  &  \qquad \hspace{2in}- V(X(s-),\al(s-))]\tilde N(ds, dz).
\end{align*}
Clearly $\ex[M_1^V(T\wedge \tau_k)]=\ex[M_2^V(T\wedge \tau_k)] =0$.  We need to analyze the term $M_3^V(\tau_k \wedge T)$ carefully.
Now   it follows from \eqref{V-lip} that
\bea \ad \ex \biggl[\int^{\tau_k \wedge T}_0 \!\! \int_{\rr^n_0} e^{-2\gamma_0 s}     \large| V(X(s-) + \gamma(X(s-),\al(s-),z),\al(s-))%\\ \aad \hfill
  - V(X(s-),\al(s-))\large |^2 \nu(dz) ds \biggl] \\ [2ex]
  \aad  \le \ex \biggl[ \int^{\tau_{k}\wedge T}_0\!\! \int_{\rr^n_0} K_V^2 \abs{ \gamma(X(s-),\al(s-),z)}^2 \nu(dz) ds\biggl]\\[2ex]
\aad \le  K_V^2  \kappa_k  \ex \biggl[  \int^{T}_0  (1+ \abs{X(s-)}^2) ds \biggl]   \le K_V^2  \kappa_k  (1+ k^2) T < \infty.
\eea
Thus it follows that $\ex[M_3^V(T\wedge \tau_k)] =0$, as desired.
Taking expectations on both sides of \eqref{V-Ito} and using \eqref{eq:reg-cond},  we have
  \bed \barray \ad \ex_{x,i}
\left[V(X(\tau_k \wedge T),\al(\tau_k  \wedge T))
\exp
(-\ga_0(\tau_k \wedge T ))\right] - V(x,i) \\
\aad \ = \ex_{x,i}  \biggl[\int^{\tau_k  \wedge T}_{0} \({\partial \over
\partial
t} + \op\) U(X(u-),\al(u-),u-) du\biggl] \le 0.\earray\eed
Hence  for all $k \ge k_{1}$,
by \eqref{2-beta-k} and \eqref{eq2:reg-cond},  we have  $$\begin{aligned}V(x,i) & \ge
 \ex_{x,i}\left[V(X(\tau_k \wedge T),\al(\tau_k  \wedge T))
\exp
(-\ga_0(\tau_k \wedge T ))\right].
\\
& \ge  \ex_{x,i}\left[ V(X(\tau_{k}),\alpha(\tau_{k}))\exp(-\gamma_{0}\tau_{k}) I_{\set{\tau_{k}\le  T}} \right]
\\ & \ge \ex_{x,i}\left[ V_{k} e^{-\gamma_{0}T}I_{\set{\tau_{k}\le  T}} \right] \\
&  > \e V_{k}e^{-\gamma_{0}T} \to \infty \ \hbox{ as } \ k \to \infty,
\end{aligned}$$  which is a contradiction. Hence we must have $\tau_{\infty} = \infty$ a.s. or the process $(X\cd,\al\cd)$ is a global solution to the system
given by \eqref{sw-jump-diffusion} and \eqref{Q-gen} (or equivalently, \eqref{sw-jump-diffusion} and \eqref{eq:jump-mg}).

We proceed to establish the uniqueness. Suppose that there is another global solution $(\tilde X, \tilde \al)$ to
\eqref{sw-jump-diffusion} and \eqref{Q-gen} (or equivalently, \eqref{sw-jump-diffusion} and \eqref{eq:jump-mg}).
By the construction of the processes $(X,\al)$ and $(\tilde X, \tilde \al)$ and the fact that $\tau_{k} \to \infty$ a.s. as $k \to \infty$, we have
We consider $$\begin{aligned}\pr&\set{(X(t),\al(t)) =(\tilde X(t),\tilde \al(t)), \forall 0 \le t < \infty } \\  & = \pr\set{\bigcap_{k=1}^{\infty}(X(t),\al(t)) =(\tilde X(t),\tilde \al(t)), \forall 0 \le t < \tau_{k} } \\
& = \lim_{k \to \infty} \pr\set{(X(t),\al(t)) =(\tilde X(t),\tilde \al(t)), \forall 0 \le t < \tau_{k} } \\
& =1. \end{aligned}$$
This shows that the solution to  \eqref{sw-jump-diffusion}--\eqref{eq:jump-mg} is pathwise unique.
\end{proof}

\section{Feynman-Kac Formula}\label{sect-feykac}

We aim to   derive several versions of the Feynman-Kac formulas, corresponding to coupled  systems of partial integro-differential  equations  of Cauchy and Dirichlet types, respectively. Section \ref{sect-cauchy} deals with the Cauchy problem and Section \ref{sec:1st} investigates the Dirichlet problem. To this end, we need the following lemma, which establishes the moment bounds for the regime-switching jump diffusion.

\begin{lem}\label{lem-bdd-moments} Let $T >0$ be fixed.
\begin{itemize}
  \item[{\em (a)}] Then under Assumption {\em (A1)},  for any positive constant $p\in (0,2]$, we have
\begin{equation}
\label{eq-momen-bd}
\ex_{x,i}\biggl[\sup_{t\in [0,T]}\abs{X(t)}^p\biggl] \le K <\infty, \ \ (x,i) \in \rr^n \times \M,
\end{equation}
where $K=K(x,T,p)$ is a constant.

\item[{\em (b)}] Suppose  Assumption {\em (A1)}.  In addition, if for some $\tilde p >2$ and $\kappa_2>0$,
 \begin{equation} \label{eq-gamma-p}\int_{\rr^n_0} \abs{\ga(x,i,z)}^{\tilde p} \nu(dz)
 \le \kappa_2(1+\abs{x}^{\tilde p}), \ (x,i) \in \rr^n \times \M.\end{equation}
 Then \eqref{eq-momen-bd}  is satisfied for all $p \in (0, \tilde p]$.
\end{itemize}
\end{lem}

\begin{proof} We shall  only prove Part (b);   Part (a)  can be established in a  similar manner.

Step 1: Consider first the case $p=\tilde p\ge  2$.
Note that
\beq{eq:jp-1}\barray
|X(t)|^{\tilde p} \ad \le 4 ^{\tilde p-1} \left[  |x|^{\tilde p}+ \l \int^{t}_0 b(X(s-),\al(s-)) ds \r ^{\tilde p} + \l \int^{t}_0 \sg (X(s-),\al(s-)) dW(s) \r^{\tilde p}\right. \\
\aad \qquad \qquad + \left. \l \int^{t}_0 \int_{\rr^n_0} \ga (X(s-),\al(s-), y) \wdt N (ds, dy)\r^{\tilde p} \right].\earray\eeq
Using \eqref{ito-condition}, taking expectation in \eqref{eq:jp-1}, for the first terms on the right-hand side of \eqref{eq:jp-1},
similar to \cite[Proposition 2.3, pp.31-33]{YZ-10}, we obtain
\beq{eq:jp-2}\barray
\ad  \ex_{x,i}\left[ \sup_{0\le t\le T}
\left\{ |x|^{\tilde p}+ \l \int^{t}_0 b(X(s-),\al(s-)) ds \r ^{\tilde p} + \l \int^{t}_0 \sg (X(s-),\al(s-)) dW(s) \r^{\tilde p}\right\}\right]
\\
\aad \ \le K_1 + K_2 \int^T_0  \ex_{x,i} \sup_{1\le u\le s}| X(u)|^{\tilde p} ds ,\earray\eeq
where $K_i=    K_i(x,\tilde p,T)  $ for $i=1,2$.

By virtue of \cite[Lemma 4]{MR10} (see also \cite[Theorem 4.4.23]{APPLEBAUM}) together with \eqref{ito-condition} and \eqref{eq-gamma-p},
\begin{align} 
 \nonumber  &  \ex_{x,i}   \left[\sup_{0\le t\le T} \biggl|\int^t_0 \int_{\rr^n_0} \ga (X(s-),\al(s-), y) \wdt N (ds, dy)\biggl|^{\tilde p} \right]
\\ \nonumber 
 &  \  \ \le K_3 \ex_{x,i}\biggl[ \int^{T}_0 \Big(\int_{\rr^n_0} |\ga (X(s-),\al(s-), y)|^2 \nu (dz)\Big)^{\tilde p/2}ds+  \int^{T}_0 \int_{\rr^n_0} |\ga (X(s-),\al(s-), y)|^{\tilde p} \nu (dz)ds\biggr] \\
%  \nonumber  &  \qquad + K_3  \ex_{x,i} \left[ \int^{T}_0 \int_{\rr^n_0} |\ga (X(s-),\al(s-), y)|^{\tilde p} \nu (dz)ds\right]\\
\nonumber  &  \ \  \le K_3 \ex_{x,i}\left[\int^{T}_0 \Big(  [1+ |X(s)|^2] \Big)^{\tilde p/2} ds\right]  + K_3  \ex_{x,i} \left[ \int^{T}_0 [1+|X(s)|^{\tilde p}] ds\right]\\
& \ \  \label{eq:jp-3} \le K_4 +K_5 \int^{T}_0\ex_{x,i} \left[\sup_{0\le u\le s} |X(u)|^{\tilde p} \right]ds. \end{align}
Again, $K_i, i=3, 4, 5$  depend  on $x$, $\tilde p$, and $T$ only.
Combining \eqref{eq:jp-2} and \eqref{eq:jp-3}, we obtain
that
\beq{eq:jp-4} \barray
\disp   \ex_{x,i}\left[ \sup_{0\le t\le T} |X(t)|^{\tilde p}\right] \ad\le K_5  + K_6 \int^T_0 \ex_{x,i}\left[\sup_{1\le u\le s} | X(u)|^{\tilde p} \right] ds,\earray\eeq
with $K_4$ and $K_5$ depending on $x$, $\tilde p$, and $T$.
The desired result then follows from Gronwall's inequality.

Step 2: The case when   $ 1\le p < \tilde p$ follows from H\"older's inequality.

Step 3: Suppose that $0<p<1$.
Since $|x|^p= |x|^p I_{\{|x|\ge 1\}} + |x|^p [1-I_{\{|x|\ge 1\}}]\le 1+|x|^{1+p}$,
using the result in Step 2,
$$\ex_{x,i} \left[ \sup_{0\le t\le T}|X(t)|^p\right]\le \ex_{x,i}\left[ 1+ \sup_{0\le t\le T}|X(t)|^{1+p}\right]\le K < \infty.$$
Combing the above steps gives \eqref{eq-momen-bd}.
\end{proof}

\subsection{Cauchy Problems}\label{sect-cauchy}

\begin{thm}\label{thm-feyman1}
Assume %\eqref{add-gamma-conds} and
 {\em (A1)}.
Consider the coupled system of partial integro-differential equations of the form
 \begin{equation}\label{feynman-kac}
\begin{cases}\dfrac{\partial}{\partial t} u(t,x,i)= \op u(t,x,i) - c(x,i) u(t,x,i),\ \ &(t,x,i)\in (0,\infty)\times \rr^n \times \M, \\
u(0,x,i)= f(x,i), \ \ & (x,i)\in  \rr^n \times \M, \end{cases}
\end{equation}
where $\op$ is as in \eqref{op-defn},   $0 \le c(\cdot,i) \in C(\rr^n )$, and $f(\cdot, i)\in C(\rr^n)$ for each $i\in \M$.
 If $u$ is a classical solution to \eqref{feynman-kac} satisfying %certain growth conditions,
\begin{equation}
\label{eq-u-growth-cond}
\abs{u(t,x,i)}  \le K (1+ \abs{x}^2), \text{ for some } K>0 \hbox{ and all } t \ge 0 \text{ and } (x,i) \in \rr^n \times \M,
\end{equation}
then it admits a stochastic representation
\begin{equation}\label{stoch-soln-FK}\begin{aligned}
u(t,x,i)  & =
  \ex_{x,i} \left[\exp\set{-\int_0^t c(X(s),\al(s))ds} f(X(t),\al(t))\right] \\
 &  = \ex \left[\exp\set{-\int_0^t c(X^{x,i}(s),\al^{x,i}(s))ds} f(X^{x,i}(t),\al^{x,i}(t))\right].
  \end{aligned} \end{equation}
\end{thm}

 \begin{rem}\label{rem-about-classical}
A smooth function $u: [0,\infty) \times\rr^n \times \M \mapsto \rr$ is said to be a {\em classical solution} of \eqref{feynman-kac} if  (i) $u(\cdot, \cdot, i) \in C^{1,2}$ for each $i\in \M$, (ii) $u(t, \cdot, \cdot)$ belongs to the domain $\mathcal D_\op$ of the generator $\op$ for each $t\in [0,\infty)$, and (iii) $u$ satisfies both equations in \eqref{feynman-kac} in the classical sense.
\end{rem}

 \begin{proof} Suppose $u$ is a classical solution to \eqref{feynman-kac}.  Let $(X\cd,\alpha\cd)$ be the unique solution to \eqref{sw-jump-diffusion} and \eqref{Q-gen} (or equivalently, \eqref{sw-jump-diffusion} and \eqref{eq:jump-mg})
 with initial condition $(X(0),\alpha(0))=(x,i)$. For any $t>0$, we define for $0\le s \le t$,   $$M(s):= \exp\set{-\int_0^s c(X(r),\al(r))dr} u(t-s, X(s),\al(s)).$$ Then it follows from generalized It\^o's formula \eqref{eq:ito} and the first equation of \eqref{feynman-kac} that
 \begin{displaymath}
\begin{aligned}
M(s)-M(0)  =& \int_{0}^{s} \exp\set{-\int_0^v c(X(r),\al(r))dr} \biggl[\(-\frac{\partial }{\partial v}   + \op \)u(t-v, X(v),\al(v)) \\
 &  \hspace{.45in}-c(X(v),\alpha(v)) u(t-v, X(v), \alpha(v)  \biggl]dv + M_{1}(s) + M_{2}(s) + M_{3}(s), \\
 =& M_1(s) + M_2(s) + M_3(s),
\end{aligned}
\end{displaymath}
where \begin{displaymath}
\begin{aligned}
 M_1(s) = &  \int_0^s  \exp\!\set{-\int_0^r c(X(v),\al(v))dv}  \\
               & \hspace{1.8in} \times D_x u(t-r, X(r-),\al(r-)) \cdot \sigma(X(r-),\al(r-)) dW(r), \\
 M_2(s) =& \int_0^s \!\!\int_{\rr}  \exp\!\set{-\int_0^r c(X(v),\al(v))dv}\big[ u(t-r, X(r-), \al(s-)+ h(X(r-),\al(r-),z))   \\ & \hspace{1.3in} -u(t-r,X(r-),\al(r-))\big]\tilde N_{1}(dr, dz), \\
 M_3(s) = & \int_0^s\!\! \int_{\rr^n_0} \exp\!\set{-\int_0^r c(X(v),\al(v))dv} \big[u(t-r, X(r-) + \gamma(X(r-),\al(r-),z),\al(r-))  \\ &     \hspace{1.3in}
 - u(t-r, X(r-),\al(r-))\big]\tilde N(dr, dz).
\end{aligned}
\end{displaymath}
 Therefore $\set{M(s), 0 \le s \le t}$ is a local martingale. Put $\tau_n:= \inf\set{t\ge 0: \abs{X(t)} \ge n}$, $ n\in \mathbb N$. Then thanks to \eqref{ito-condition},  \eqref{eq-u-growth-cond}, and  Lemma \ref{lem-bdd-moments},  the processes $M_i(\cdot \wedge \tau_n), i =1, 2, 3$, are martingales; so is the process $M(\cdot\wedge \tau_n)$. Therefore, by the definition of $M$ and the optional sampling theorem, we obtain
\begin{displaymath}
\ex[M(s\wedge \tau_n)] =\ex[M(0\wedge \tau_n)]= u(t,x,i), \ \ \text{ for all } s\in [0,t].
\end{displaymath}
Owing to  \eqref{eq-u-growth-cond}, the definition of the process $M$, and the assumption that $c\ge 0$, for any $s\in [0, t]$, we have
$$|M(s\wedge \tau_n)| \le K\(1+ |X(s\wedge \tau_n)|^2\) \le K\(1+ \sup_{0\le s \le t}|X(s)|^2\).$$
Therefore by letting $n\to \infty$, we obtain from  the dominated convergence theorem and Lemma \ref{lem-bdd-moments}   that ${\mathbb E}[M(s)]= u(t,x,i), s\in[0,t]$.  In particular,
\eqref{stoch-soln-FK} follows.
% Owing to \eqref{eq-u-growth-cond} and  Lemmas \ref{lem-bdd-moments}, for any  $s\in [0,t]$ and  $n \in \N$,  we have $$\ex[|M(s\wedge \tau_{n})|] \le K \(1+ \ex\biggl[\sup_{0\le s \le t} |X(s)|^{2}\biggl]\) < \infty .$$
%Therefore by letting $n\to \infty$,
%the dominated convergence theorem yields that $\ex[M(s)] = u(t,x,i), s\in [0,t]$.   In particular,
%\eqref{stoch-soln-FK} follows.
 \end{proof}

\begin{thm}\label{thm-FK-terminal}
Assume {\rm (A1)} and for some $T >0$. Suppose that $u(\cdot,\cdot,i): [0,T]\times \rr^{n}\times \M \mapsto \rr$ is of class $C^{1,2}([0,T) \times \rr^{n}) \cap C([0,T]\times \rr^{n})$ for each $i \in \M$ and  $u(t, \cdot, \cdot) \in \mathcal D_{\op}$ for every $t\in [0,T]$ and satisfies the Cauchy  problem
\begin{equation}
\label{eq-Cauchy-terminal}
\left\{\!\!\begin{array}{rlll}
 \disp  \frac{\partial }{\partial t} u(t,x,i) + \op u(t,x,i) - c(t,x,i) u(t,x,i) & \!=&\! g(t,x,i),     &  (t,x,i) \in [0,T) \times \rr^{n} \times \M, \\
     u(T,x,i) & \! =& \! f(x,i) , &     (x,i) \in \rr^{n }\times \M,
\end{array}\right.
\end{equation}
and that satisfies the growth condition \begin{equation}
\label{eq2-u-growth-cond}
\abs{u(t,x,i)}  \le K (1+ \abs{x}^p), \text{ for all } t \in [0,T] \text{ and } (x,i) \in \rr^n \times \M, \quad 0 \le p < 2,
\end{equation} where
for each $i\in \M$, the functions $c(\cdot,\cdot, i) \ge 0$, $g (\cdot,\cdot, i)$, and $f(\cdot,i)$ are continuous and satisfy
\begin{equation}
\label{eq-g-growth-cond}
|g(t,x,i) | + | f(x,i) |\le K (1+ |x|^{2}), \quad \forall t \in[0,T], x\in \rr^{n},
\end{equation} and for some $K>0$.
Let $(X,\alpha)= (X^{t,x,i},\alpha^{t,x,i})$ be the solution to
\eqref{sw-jump-diffusion} and \eqref{Q-gen} $($or equivalently, \eqref{sw-jump-diffusion} and \eqref{eq:jump-mg}$)$
with $(X(t),\alpha(t)) =(x,i)$. Then we have % $u$ admits the stochastic representation
\begin{equation}
\label{eq-FK-stoch-soln-terminal}
\begin{aligned}
u(t,x,i) = &  \ex_{t,x,i}   \biggl[ e^{-\int_{t}^{T} c(r, X(r),\alpha(r)) dr } f(X(T),\alpha(T))    \\ &  \qquad -   \int_{t}^{T} e^{-\int_{t}^{s} c(r, X(r),\alpha(r)) dr }  g(s,X(s),\alpha(s)) ds   \biggl], \quad  0\le t \le T.
\end{aligned}\end{equation}
\end{thm}

\begin{proof} Define $\tau_{n}$ as in the proof of Theorem  \ref{thm-feyman1}. Then as before, we apply It\^o's formula to the process $u(s,X(s),\alpha(s)) \exp\set{-\int_{t}^{s} c(r, X(r),\alpha(r)) dr}, s \in [t,T]$ and then take expectations to obtain
\begin{equation}\label{eq-u-3terms}
\begin{aligned}
 u(t,x,i)  =     & - \ex_{t,x,i} \left[ \int_{t}^{\tau_{n}\wedge T} e^{- \int_{t}^{s} c(r, X(r), \alpha(r)dr}g(s,X(s),\alpha(s)) ds\right]  \\
    &  + \ex_{t,x,i}\left[  e^{-\int_{t}^{T}c(r, X(r),\alpha(r)) dr }f(X(T),\alpha(T))  I_{\set{\tau_{n}  >  T}} \right] \\
    & + \ex_{t,x,i}\left[  e^{-\int_{t}^{\tau_{n}}c(r, X(r),\alpha(r)) dr } u(\tau_{n}, X(\tau_{n}), \alpha(\tau_{n}))I_{\set{\tau_{n }  \le T}} \right].
\end{aligned}
\end{equation}
 Since $c\cd \ge 0$ and  $\tau_{n} \to \infty $ a.s. as $n \to \infty$, the first term of \eqref{eq-u-3terms} converges to $$ - \ex_{t,x,i} \left[ \int_{t}^{  T} e^{- \int_{t}^{s} c(r, X(r), \alpha(r)dr}g(s,X(s),\alpha(s)) ds\right] $$  by \eqref{eq-g-growth-cond}, Lemma \ref{lem-bdd-moments}, and the dominated convergence theorem.
 Similarly, we can show that as $n \to \infty$, the second term of \eqref{eq-u-3terms} converges to
$$ \ex_{t,x,i}\left[  e^{-\int_{t}^{T}c(r, X(r),\alpha(r)) dr }f(X(T),\alpha(T)) \right].$$ Let
us analyze the third term of  \eqref{eq-u-3terms}. Owing to \eqref{eq2-u-growth-cond}, it is bounded in absolute value by  $$ \ex_{t,x,i} \left[ \abs{u(\tau_{n}, X(\tau_{n}), \alpha(\tau_{n})) }I_{\set{\tau_{n }  \le T}} \right]  \le K \ex_{t,x,i} \left[ \abs{X(\tau_{n}\wedge T)}^{p} I_{\set{\tau_{n }  \le T}}  \right] + K \pr_{t,x,i} \set{\tau_{n} \le T}. $$ Certainly we have $\pr_{t,x,i} \set{\tau_{n} \le T} \to 0$ as $n \to \infty$.  Furthermore, since $p < 2$ in \eqref{eq2-u-growth-cond}, by the H\"older inequality, Lemma \ref{lem-bdd-moments}, and the Chebyshev inequality,  we have
 $$\begin{aligned}\ex_{t,x,i}\left[  \abs{X(\tau_{n}\wedge T)}^{p} I_{\set{\tau_{n }  \le T}}   \right] &   \le  \(\ex_{t,x,i} [\abs{X(\tau_{n} \wedge T)}^{2}]\)^{p/2} \( \pr_{t,x,i} \set{\tau_{n} \le T}\)^{ (2-p)/2} \\
 & \le  \(\ex_{t,x,i} \biggl[\sup_{0\le s \le T}\abs{X(s)}^{2}\biggl]\)^{p/2} \! \!\(  \pr_{t,x,i} \biggl\{\sup_{0\le s \le T} |X(s)| \ge n\biggl\}\)^{(2-p)/2} \\
 & \le K_{1} \( \frac{\ex_{t,x,i} \Bigl [ \sup_{0\le s \le T} |X(s)|^{2}\Bigl]}{n^{2}} \)^{(2-p)/2} \\
 & \le  K_{2} n^{p-2} \to 0 \end{aligned}$$  as    $n \to \infty$, where $K_{1}$ and $K_{2} $ are positive constants  independent of $n$. Thus the third term of \eqref{eq-u-3terms} goes to $0$ as $n \to \infty$. Finally we obtain \eqref{eq-FK-stoch-soln-terminal} by combining the above estimates into \eqref{eq-u-3terms}.
\end{proof}

\begin{rem}\label{rem32-differences} If the L\'evy measure $\nu \cd\equiv 0$,  then the process $X\cd$ is
 in fact a regime-switching diffusion and has continuous sample paths \cite{YZ-10}. In such a case, similar to \cite[Theorem 5.7.6]{Karatzas-S}, we can relax  assumption \eqref{eq2-u-growth-cond}  to
\begin{equation}
\label{eq3-u-growth-cond}
   \abs{u(t,x,i)} \le K (1+ |x|^\mu), \text{ for all }(t,x,i) \in [0,T] \times \rr^n\times \M,
\end{equation}where $\mu\ge 2$. Indeed, since
 $|X(\tau_{n})| = n$,  one can bound the third term of \eqref{eq-u-3terms}
 by $K (1+ n^\mu) \pr_{t,x,i} \{\tau_{n} \le T\}$. Furthermore,
  in view of \cite[Proposition 2.2.3]{YZ-10}, for any $\iota > \mu$, we have $$ \pr_{t,x,i} \{\tau_{n} \le T\} = \pr_{t,x,i}\set{\sup_{0\le s \le T} |X(s)| \ge n} \le \frac{\ex_{t,x,i}\left[\sup_{0\le s \le T} |X(s)|^{ \iota}\right]}{n^{\iota}} \le K n^{-\iota},$$ where $K>0$ is independent of $n$.
This shows that the third term of \eqref{eq-u-3terms}  converges to $0$ as $n\to \infty$.
However, in the presence of a
L\'evy measure $\nu\cd$,  the inequality $|X(\tau_{n})| \le n$ is not necessarily true. Thus in general  we cannot relax \eqref{eq2-u-growth-cond} and apply the arguments of \cite{Karatzas-S} directly.

If we relax \eqref{eq2-u-growth-cond} to  \eqref{eq3-u-growth-cond} for some $\mu \ge 2$ and  suppose also \eqref{eq-g-growth-cond} is replaced by \begin{equation}
\label{eq2-g-growth-cond}
|g(t,x,i) | + | f(x,i) |\le K (1+ |x|^{\tilde p}), \quad \forall t \in[0,T], x\in \rr^{n}, i \in \M,
\end{equation} where $K  $ is a  positive constant and $\tilde p > \mu$.  Then the stochastic representation   \eqref{eq-FK-stoch-soln-terminal} is still valid  if   we   assume in addition that  \eqref{eq-gamma-p} holds  for  $\tilde p > \mu$.

In fact, if   \eqref{eq-gamma-p} holds for  $\tilde p > \mu$, then Lemma \ref{lem-bdd-moments}, part (b) reveals that we still have the moment bound $\ex[\sup_{0\le s \le T} |X(s)|^{\tilde p}] \le K < \infty$, where $K= K(T,\tilde p)$. Thus similar to the proof of Theorem \ref{thm-FK-terminal}, we compute  $$ \ex[|X(\tau_{n} \wedge T)|^{\mu} I_{\set{\tau_{n} \le T}}] \le \(\ex[|X(\tau_{n}\wedge T)|^{\tilde p}]\)^{\mu/\tilde p}  \(\pr\set{\tau_{n} \le T}\)^{(\tilde p - \mu)/ \tilde p} \le K n^{-(\tilde p -\mu)} \to 0,$$ as $n \to \infty$.
This, together with  almost the same argument as that in the proof of Theorem \ref{thm-FK-terminal}, helps us to establish
 the stochastic representation \eqref{eq-FK-stoch-soln-terminal}. We summarize the above discussion into the following corollary.
\end{rem}

\begin{cor}
Assume   {\em (A1)}. Suppose   that \eqref{eq3-u-growth-cond}, \eqref{eq2-g-growth-cond}, and \eqref{eq-gamma-p} hold for some positive constants $\tilde p > \mu \ge 2$. Then the conclusion of Theorem  \ref{thm-FK-terminal} continues to hold.
\end{cor}

\begin{exm} In this example, we demonstrate that Theorem \ref{thm-FK-terminal} can be applied in mathematical finance.
We consider a generalized Black-Scholes market that consists of two assets: a bond and a stock. The price of the bond evolves according to the equation
\begin{equation}
\label{bond}
dB(t)= r(\alpha(t))B(t)dt, \ \ B(0) = 1,
\end{equation} where $\al\cd$ is a continuous-time Markov chain with generator $Q= (q_{ij})$ and a finite state space $\M = \set{1, \dots, m}$, and  $r: \M \mapsto \rr_{+}$.  Hence  the discounted process is given by  $ D(t) = \exp\{- \int_{0}^{t} r (\alpha(r))dr\}.$
Suppose that under some risk neutral measure $\mathbb Q$, the price of  the stock is modeled by the stochastic differential equation
\begin{equation}
\label{stock} \begin{cases}
dS(t)=S({t-}) \left[ r(\alpha({t-})) dt + \sigma( \alpha({t-}))  dW^{\Q}(t) +  \disp   \int_{\rr_0}  \gamma(  \alpha({t-}),  z) \tilde N^{\mathbb Q}(dt, dz)\right], & \\
S(0) =s_{0}> 0, & \end{cases}
\end{equation} where  $W^{\Q}$ is a one-dimensional Brownian motion and $ N^{\Q} $ is a Poisson random measure with compensator $\tilde N^{\Q} (t, E) :=  N^{\Q} (t,E) - t \nu^{\Q} (E) $ under the risk neutral measure $\Q$, in which $\nu^{\Q}$ is a L\'evy measure satisfying $\int_{\rr_{0}}
(1 \wedge |z|^{2}) \nu^{\Q} (dz) < \infty$.  Such a risk neutral measure $\Q$ can be found using the Esscher transform (\cite{Elliott,gerber1994option}). For simplicity, we assume that $\alpha\cd$, $W^{\Q}\cd$, and $N^{\Q}\cd$ are independent. For $i \in \M$, $r_{i}= r(i), \sigma_{i}=\sigma(i)$ are positive constants and $\gamma(i, \cdot)$ is a real-valued function satisfying $\int_{\rr_{0}} \gamma(i,z)^{2}\nu^{\Q} (dz) \le K < \infty$. Furthermore, we assume that $\gamma (i ,z) > -1$ for all $i \in \M$ and $z \in \rr_{0}$.

Note that for this example, assumption (A1) is satisfied. Therefore \eqref{stock} has a unique strong solution and  the moment bound \eqref{eq-momen-bd} holds.
Using the generalized It\^o formula \eqref{eq:ito}, we obtain $S(t) = S(0)  \exp(X(t)),$ where
$$\begin{aligned}X(t) = &  \int_{0}^{t} \! \left[ r(\alpha(v-)) - \frac{1}{2} \sigma^{2}(\alpha(v-)) + \int_{\rr_{0}} [\log(1+ \gamma(\al(v-),z)) - \gamma(\al(v-),z)]  \nu^{\Q} (dz) \right]\! dv \\
 & + \int_{0}^{t} \sigma(\al(v-)) d W^{\Q}(v) + \int_{0}^{t} \int_{\rr_{0}}   \log(1+  \gamma(\al(v-),z))   \tilde N^{\Q} (dv, dz).\end{aligned}$$

For a European type contingent claim with payoff $h(S(T),\al(T))$ at the time of expiration $T> 0$, according to the fundamental theorem  of asset  pricing \cite{delbaen1994general},  $$V_{t} := \ex_{\Q}\left[e^{-\int_{t}^{T} r(\alpha(r))dr} h(S(T),\al(T)) \big| \F_{t}\right]$$ gives an arbitrage free price of the derivative at time $t\in [0, T]$, where $\mathbb F= \set{\F_{t},  0\le t \le T}$ is the natural filtration of $S$.  Suppose that for each $i\in \M$,  the function $h(\cdot,i)$ is continuous
 and satisfies the growth condition $0 \le  h(s,i)  \le K (1+ s^{2}) $ for
  some $K>0$ and all $(s,i) \in (0, \infty) \times \M$. By the Markov property, we can write $V_{t} = V(t, S(t), \al(t))$, where \begin{equation}
\label{eq-V-option-price}
 V(t, s, i) =\ex_{\Q}\left[e^{-\int_{t}^{T} r(\alpha(v))dv} h(S(T),\al(T)) \big| S(t) = s, \al(t) = i\right].
\end{equation}

By virtue of Theorem \ref{thm-FK-terminal}, any solution satisfying the growth condition \eqref{eq2-u-growth-cond} of the terminal value problem
\begin{equation}
\label{eq-pide-example}
\left\{\!\!\begin{array}{rlll}\disp
   \frac{\partial }{\partial t} u(t,s,i) + \op^{\Q} u(t,s,i) - r_{i} u(t,s,i) & \!=&\! 0,     &  (t,s,i) \in [0,T) \times (0, \infty)  \times \M, \\
     u(T,s,i) & \! =& \! h(s,i) , &     (s,i) \in  (0, \infty) \times \M,
\end{array}\right.
\end{equation} can be represented by the right hand side of \eqref{eq-V-option-price}, where $\op^{\Q}$ is defined as follows: $$
\begin{aligned}\op^{\Q} f(s,i) = & D_{s} f(s,i)  r_{i } s +  \frac{1}{2}D_{ss} f(s,i) \sigma_{i}^{2} s^{2 } \\ &   + \int_{\rr_{0}} [f(s+ s \gamma(i,z),i) -f(s,i) - D_{s} f(s,i) s \gamma(i,z)] \nu^{\Q} (dz),\end{aligned}
$$ in which $D_{s}f(s,i)$ and $D_{ss} f(s,i) $ denote the first and second order partial derivatives of $f$ with respect to the variable $s$, respectively.  On the other hand, if one can show that the function $V$ defined in \eqref{eq-V-option-price} is a classical (or even a viscosity) solution to \eqref{eq-pide-example}, then we can find the arbitrage free price $V(t, S(t),\al(t))$ by solving \eqref{eq-pide-example}.
\end{exm}

 \subsection{A Dirichlet Problem}\label{sec:1st}
 In this subsection, we work with a bounded and open domain $D \subset \rr^{n}$ with  boundary $\partial D$ and closure  $\lbar D= D \cup \partial D$.
  Similar to Section \ref{sect-cauchy}, our goal is to obtain a stochastic representation for the classical  solution $u$ (in the sense of Remark \ref{rem-about-classical}) to the Dirichlet  problem
 \begin{equation}\label{bvp} \begin{cases}
  \op u(x,i)  -c(x,i) u(x,i) = \xi(x,i), &  (x,i) \in D\times \M,  \\
  u(x,i)= \eta(x,i), &   (x,i) \in   D^c \times \M,
\end{cases}\end{equation}
 where
 for each $i\in \M$,
 the functions $c(\cdot,i) \ge 0$, $\xi(\cdot, i)$, and $\eta(\cdot,i) $ are continuous on their domains.    To proceed, we first present the following lemma, which asserts that starting from a point $(x,i) \in D \times \M$, the regime-switching jump diffusion given by
 \eqref{sw-jump-diffusion} and \eqref{Q-gen} (or equivalently, \eqref{sw-jump-diffusion} and \eqref{eq:jump-mg})
 will exit from $D$ in finite time a.s.

\begin{lem}\label{lem-3.8}
\label{lem-tau-finite} Assume {\em (A1)}. In addition,  suppose that for some $\ell \in \set{1, 2, \dots,n}$,   \begin{equation}
\label{eq-lem3.8-cond-a>0}
 \wdt a := \min_{x \in \lbar D, i\in \M} a_{\ell\ell} (x,i) > 0.
\end{equation}
Let  $\tau$ be the first exit time of the regime-switching jump
diffusion of   \eqref{sw-jump-diffusion} and \eqref{Q-gen} $($or equivalently, \eqref{sw-jump-diffusion} and \eqref{eq:jump-mg}$)$
from $D$:   $$\tau: = \inf\{t\ge 0: X^{x,i}(t) \not
\in D \}.$$    Then $\tau <
\infty$ a.s. and $\ex[\tau] < \infty$.
\end{lem}

\begin{proof}
 We proceed to
 use the idea in \cite{wee2000recurrence}.
Since $D$ is bounded, there exists an $R>0$ such that $|x| \le R$ for all $x\in D$. Furthermore, let us choose positive constants $\beta > 3R$ and $K > \beta^{2p}$,
where $p>1$ is
  to be specified later.
  Consider a function $\vphi\in C_{c}^{2}(\rr)$ such that $$ \vphi(z) = \begin{cases} K - z^{2p}, &\text{ if }\  |z| \le \beta, \\  0, & \text{ if }\ |z| \ge \beta +1. \end{cases}$$ Moreover, we can pick $\vphi$ so that it is nonnegative, symmetric about $0$, and nonincreasing on $[0,\infty)$. For any $(t,x,i) \in [0,\infty) \times \rr^{n} \times \M, $ set $\Phi(t,x,i):= e^{\lambda t} \phi(x,i)$ with $\phi(x,i):= \vphi(x_{\ell} - 2R)$, where $x_{\ell}$ is the $\ell$th component of $x$ and  $\lambda >0$ is a constant to be determined.
Note that $\phi \in \mathcal D_{\op}$ and that  for all $x\in \lbar D$, we have \begin{equation}
\label{eq-lem-3.8-x-ell-2R}
  -3R = - R -2R\le  x_{\ell} -2R \le R -2 R = - R.
\end{equation} Hence $\abs{x_{\ell} - 2R}  \le 3R < \beta$.   Thus for all $(x,i) \in  D \times \M, $ $\phi(x,i) =   K- (x_{\ell} - 2R)^{2 p}$, which, in turn, implies that
\begin{equation}
\label{eq-lem3.8-op}
  \begin{aligned}\op \phi(x,i)  = &   -2 p (x_{\ell} -2R)^{  2p -1} b_{\ell } (x,i)  - \frac{1}{2} 2 p (2p-1) (x_{\ell}- 2R)^{ 2 p-2} a_{\ell\ell}(x,i) \\ & + \int_{\rr_{0}^{n}}   \left[ \phi(x+ \gamma(x,i,z), i )  - \phi(x,i) - D_{x} \phi(x,i) \cdot \gamma(x,i,z) \right] \nu(dz).  \end{aligned}
\end{equation}We claim that for all $(x,i) \in D \times \M$, \begin{equation}
\label{eq-claim-lem3.8} \begin{aligned}
\int_{\rr_{0}^{n}} &   \left[\phi(x+ \gamma(x,i,z), i )  - \phi(x,i)  - D_{x} \phi(x,i)  \cdot \gamma(x,i,z) \right] \nu(dz) \le   2 p |x_{\ell} -2R|^{2p-1}   \frac{\kappa (1+ R^{2})}{\beta -3 R}.
   %\int_{\set{| x_{\ell}+ \gamma_{\ell} (x,i,z) -2R| > \beta}} \gamma_{\ell}(x,i,z)\nu(dz)   .
\end{aligned}\end{equation}
To see this, we consider two cases. If $|x_{\ell} + \gamma_{\ell}(x,i,z) -2 R| \le \beta$, then $\phi(x+ \gamma(x,i,z),i) = \vphi(x_{\ell} + \gamma_{\ell}(x,i,z) - 2R) =  K- (x_{\ell} + \gamma_{\ell}(x,i,z)-2R)^{2p}$. Using
a Taylor expansion, we have
$$ \begin{aligned}
 \phi & (x+ \gamma(x,i,z), i )  - \phi(x,i)  - D_{x} \phi(x,i)  \cdot \gamma(x,i,z)  \\
  &  =   \vphi (x_{\ell} -2R+ \gamma_{\ell}(x,i,z)) - \vphi (x_{\ell}-2R)   -
   \vphi_{x_\ell} (x_{\ell} - 2R ) \gamma_{\ell}(x,i,z) \\
    &  = \frac{1}{2}
    %\vphi''
    \vphi_{x_\ell x_\ell} ( x_{\ell} - 2R + \theta \gamma_{\ell }(x,i,z)) (\gamma_{\ell}(x,i,z))^{2},
 \end{aligned}$$ for some  $\theta \in [0,1]$, where
  $\vphi_{x_\ell} =(d/ dx_\ell) \vphi$ and $\vphi_{x_\ell x_\ell}= (d^2/d x_\ell^2)\vphi$
 are the first-order and second-order derivatives of $\vphi$, respectively.  Since  $|x_{\ell} - 2R| \le \beta$ and $|x_{\ell} + \gamma_{\ell}(x,i,z) -2 R| \le \beta$, we must have   $| x_{\ell} - 2R + \theta \gamma_{\ell }(x,i,z)| \le \beta$. Thus
   $$\begin{aligned}
    \vphi_{x_\ell x_\ell} ( x_{\ell} - 2R + \theta \gamma_{\ell }(x,i,z))  =  -2p (2p-1)( x_{\ell} - 2R + \theta \gamma_{\ell }(x,i,z))^{2p-2}   \le 0.
\end{aligned}$$ Therefore we arrive at\begin{equation}
\label{eq1-claim-lem3.8}
   \phi   (x+ \gamma(x,i,z), i )  - \phi(x,i)  - D_{x} \phi(x,i)  \cdot \gamma(x,i,z)  \le 0 \text{ if }|x_{\ell} + \gamma_{\ell}(x,i,z) -2 R| \le \beta.
\end{equation}

On the other hand, if $|x_{\ell} + \gamma_{\ell}(x,i,z) -2 R| > \beta$,
 by the monotonicity of $\vphi$ on $[0, \infty)$, we have $$\phi(x+ \gamma(x,i,z),i) = \vphi(x_{\ell} + \gamma_{\ell}(x,i,z) - 2R)  = \vphi(\abs{x_{\ell} + \gamma_{\ell}(x,i,z) - 2R})  \le \vphi(\beta).$$
Note that for all $(x,i) \in D \times \M$,  since $|x_{\ell} -2 R| \le 3 R < \beta $, $$\phi (x,i) = \vphi(x_{\ell} - 2R)= \vphi(|x_{\ell}- 2R|)  \ge \vphi(\beta).$$ Then it follows that
\begin{equation}
\label{eq2-claim-lem3.8}
\begin{aligned}& \int_{\set{| x_{\ell}+ \gamma_{\ell} (x,i,z) -2R| > \beta}}    \left[\phi(x+ \gamma(x,i,z), i )  - \phi(x,i)  - D_{x} \phi(x,i)  \cdot \gamma(x,i,z) \right] \nu(dz)\\
& \ \ \le - \int_{\set{| x_{\ell}+ \gamma_{\ell} (x,i,z) -2R| > \beta}} D_{x} \phi(x,i)  \cdot \gamma(x,i,z) \nu (dz) \\
& \ \ =  \int_{\set{| x_{\ell}+ \gamma_{\ell} (x,i,z) -2R| > \beta}}  2p (x_{\ell} -2R)^{2p-1} \gamma_{\ell} (x,i,z) \nu(dz).\end{aligned}
\end{equation} Furthermore, a simple contradiction argument reveals that for all $(x,i) \in D \times \M$,  $$\set{z \in \rr_{0}^{n}: | x_{\ell}+ \gamma_{\ell} (x,i,z) -2R| > \beta}  \subset \set{z \in \rr_{0}^{n}: |\gamma_{\ell}(x,i,z)| > \beta - 3 R}.$$ Then it follows from \eqref{ito-condition} that
$$  \begin{aligned} &  \int_{\set{| x_{\ell}+ \gamma_{\ell} (x,i,z) -2R| > \beta}}  2p (x_{\ell} -2R)^{2p-1} \gamma_{\ell} (x,i,z) \nu(dz) \\ & \ \  \le \int_{\set{| x_{\ell}+ \gamma_{\ell} (x,i,z) -2R| > \beta}} 2 p |x_{\ell} -2R|^{2p-1} |\gamma_{\ell} (x,i,z)|  \nu(dz) \\
& \ \ \le  2 p |x_{\ell} -2R|^{2p-1} \int_{\set{ |\gamma_{\ell}(x,i,z)| > \beta - 3 R}}  |\gamma_{\ell} (x,i,z)|  \nu(dz)  \\
& \ \ \le 2 p |x_{\ell} -2R|^{2p-1} \int_{\set{ |\gamma_{\ell}(x,i,z)| > \beta - 3 R}}  \frac{|\gamma_{\ell}(x,i,z)|^{2}}{ \beta -3R} \nu (dz) \\
& \ \ \le  2 p |x_{\ell} -2R|^{2p-1}   \frac{\kappa (1+ |x|^{2})}{\beta -3 R}  \le 2 p |x_{\ell} -2R|^{2p-1}   \frac{\kappa (1+ R^{2})}{\beta -3 R}. \end{aligned} $$
This, together with \eqref{eq1-claim-lem3.8} and \eqref{eq2-claim-lem3.8}, implies \eqref{eq-claim-lem3.8}.
Putting  \eqref{eq-claim-lem3.8} into \eqref{eq-lem3.8-op},  and using \eqref{eq-lem3.8-cond-a>0} and \eqref{eq-lem-3.8-x-ell-2R}, we deduce
\begin{equation}
\label{eq-lem3.8-op-phi}
  \begin{aligned}\op \phi(x,i) &  \le - 2 p (x_{\ell} -2R)^{2p-2} \! \left[ (x_{\ell}- 2R) b_{\ell}(x,i) + \frac{1}{2} a_{\ell \ell}(x,i) (2p-1) + (x_{\ell} -2R) \frac{\kappa (1+ R^{2})}{\beta - 3 R}\!\right] \\
 & \le  -2 p R^{2p-2}  \left[\frac{1}{2} \tilde a (2p-1)  + \min_{x\in \lbar D} \set{ (x_{\ell}- 2R) b_{\ell}(x,i) + (x_{\ell} -2R) \frac{\kappa (1+ R^{2})}{\beta - 3 R}}\right] \\
 & < -\rho < 0, \end{aligned}
\end{equation} by selecting $p > 1$ sufficiently large, where $\rho = \rho(p, R, \beta, \lbar D) > 0 $ is a constant.

 Now we apply the generalized It\^o formula to the process $\Phi(X(t),\al(t)) = e^{\lambda t} \phi(X(t),\alpha(t))$ to get
 \begin{displaymath}
\begin{aligned}
 e^{\lambda ( t\wedge \tau)} \phi(X(t\wedge \tau),\al(t\wedge \tau))    & = \phi(x,i) + \int_{0}^{t\wedge \tau} \!\!e^{\lambda s} [\lambda \phi(X(s-),\alpha(s-)) + \op \phi(X(s-),\al(s-)]ds    \\
    &   \ \quad + M_{1}^{\Phi}(t\wedge \tau) + M_{2}^{\Phi}(t\wedge \tau) + M_{3}^{\Phi}(t\wedge \tau),
\end{aligned}
\end{displaymath} where $M_{k}^{\Phi}(t), k=1,2,3$ are  defined as
%similarly as in the proof of Lemma \ref{lem-reg-0}:
$$ \begin{aligned} & M_1^{\Phi}(t)  = \int^{t}_0  e^{\lambda s}D_x
\phi (X(s-),\al(s-)) \cdot \sigma(X(s-),\al(s-))
d W(s), \\
 & M_2^{\Phi}(t)  =  \int^{t}_0\!\!\int_{\rr} e^{\lambda s} \big[ \phi( X(s-), \al(s-)+
h(X(s-),\al(s-),z))
-\phi (X(s-),\al(s-))\big] \tilde N_{1}(ds, dz), \\
&  M_3^{\Phi}(t) = \int^{t}_0 \!\! \int_{\rr^n_0} e^{\lambda  s} [ \phi(X(s-) + \gamma(X(s-),\al(s-),z),\al(s-))
%\\ &  \qquad \hspace{2in}
- \phi(X(s-),\al(s-))]\tilde N(ds, dz).
\end{aligned}$$
Using the definition of $\phi\cd$, we have $\ex[M_k^\Phi(t \wedge \tau)]  =0$ for $k=1,2,3$.
Moreover, for all $s \ge 0$, $\phi(X(s-),\al(s-))  \le K < \infty$, with $K$ being the constant in the definition of the function $\vphi\cd$. Note that for  $\lambda > 0$ sufficiently small, we have $\lambda K - \rho  < 0$. Thus  it follows from \eqref{eq-lem3.8-op-phi} that
\begin{displaymath}\begin{aligned}
\ex& [e^{\lambda ( t\wedge \tau)} \phi(X(t\wedge \tau),\al(t\wedge \tau)) ] \\ & =  \phi(x,i) +  \ex\left[ \int_{0}^{t\wedge \tau} \!\!e^{\lambda s} [\lambda \phi(X(s-),\alpha(s-)) + \op \phi(X(s-),\al(s-)]ds\right] \\
& \le  \phi(x,i) +  \ex\left[ \int_{0}^{t\wedge \tau} \!\! e^{\lambda s} [\lambda K -  \rho ] ds \right] \\
& = \phi(x,i) - (\rho - \lambda K) \frac{\ex[e^{\lambda (t \wedge \tau)}] - 1}{ \lambda }.
\end{aligned}\end{displaymath}
Now since the function $\phi\cd$ is nonnegative and bounded above by $K$, we have
\begin{displaymath}\begin{aligned}
\ex[e^{\lambda (t \wedge \tau)}]  & \le \frac{\lambda \phi(x,i)}{\rho -\lambda K} + 1 - \frac{\lambda}{\rho -\lambda K}\ex[e^{\lambda ( t\wedge \tau)} \phi(X(t\wedge \tau),\al(t\wedge \tau))] \\
& \le \frac{\lambda K }{\rho -\lambda K} + 1.
\end{aligned}\end{displaymath}
Passing to the limit as $t \to \infty$ yields that $\ex[ e^{\lambda \tau}] < \infty$.
Thus, we conclude that $\tau< \infty$ a.s. and $\ex[\tau] < \infty$.
 \end{proof}

To proceed, we  further  impose the following condition.

\begin{itemize}
\item[(A2)] Assume
 $\partial D \in C^2$ and that for each $i\in \M$,
 \begin{displaymath}
\begin{cases}
         c(x,i)\ge  0 \text{ and } c(\cdot,i) \text{ is uniformly H\"older continuous
in }\lbar D, \\
         \xi(\cdot,i)  \text{  is uniformly continuous in } \lbar D \text{ and
 } \eta(\cdot,i)
\text{ is continuous and bounded on }   D^c.
\end{cases}
\end{displaymath}
\end{itemize}

\begin{thm}\label{thm-fk-bdd-region}  Assume % the conditions of \lemref{lem-tau-finite}
%Assumptions
{\rm (A1)}, {\em (A2)}, \eqref{eq-lem3.8-cond-a>0}, and
%Assumption  {\rm (A3)}.
\begin{equation}
\label{eq-thm3.9-new-condtion}
\sup\set{\abs{\gamma(x,i,z)}: (x,i,z) \in \lbar D \times \M \times \rr_{0}^{n}} \le K < \infty.
\end{equation}
 Then the solution of the
system of boundary value
problem \eqref{bvp}  is given by
\begin{equation}\label{fk}\barray
 u(x,i) \ad = \ex_{x,i} \left[\eta (X(\tau), \al(\tau))
\exp \(- \int^\tau_0 c(X   (s),\al    (s)) ds \) \right] \\
\aad \quad - \ex_{x,i}
\left[ \int^\tau_0 \xi(X   ( t-),\al   (t-)) \exp \(-
\int^t_0 c(X   (s),\al    (s)) ds \) dt \right] .\earray\end{equation}
\end{thm}

\begin{proof}
We consider the process $M_{t}:= \exp\{-\int_{0}^{t}c(X(s),\al(s))ds\} u(X(t),\al(t))$.  Since $u \in
{\mathcal D}_{\op}$
for any $t\ge 0$, we can apply the generalized It\^o formula \eqref{eq:ito} to obtain
\begin{displaymath}
\begin{aligned}
   \exp& \set{-\int_{0}^{t\wedge \tau}c(X(s),\al(s))ds} u(X(t\wedge \tau),\al(t\wedge \tau)) - u(x,i) \\
    & = \int_{0}^{t\wedge \tau}\!\! e^{-\int_{0}^{s} c(X(r),\al(r))dr} \left[ -c(X(s-),\al(s-)) u(X(s-),\al(s-))  + \op u(X(s-),\al(s-))\right] ds  \\
    &  \ \ + M_{1}(t\wedge \tau) + M_{2}(t\wedge \tau) +M_{3}(t\wedge \tau),
\end{aligned}
\end{displaymath} where
\begin{displaymath}
\begin{aligned}
 M_{1}(t)    & = \int_{0}^{t}    e^{-\int_{0}^{s} c(X(r),\al(r))dr}  D_{x} u(X(s-),\al(s-)) \cdot \sigma(X(s-),\al(s-)) dW(s),\\
  M_{2}(t)   &  = \int_{0}^{t} \int_{\rr}  e^{-\int_{0}^{s} c(X(r),\al(r))dr}  \\ & \hspace{.8in} \times  [ u(X(s-), \alpha(s-)  + h(X(s-),\al(s-),z)) - u(X(s-),\al(s-))]  \tilde N_{1}(ds, dz),
  \\M_{3}(t)  & =    \int_{0}^{t} \int_{\rr_{0}^{n}}  e^{-\int_{0}^{s} c(X(r),\al(r))dr}  \\ & \hspace{.75in} \times  [ u(X(s-)+\gamma(X(s-),\al(s-),z), \alpha(s-)    ) - u(X(s-),\al(s-))]  \tilde N (ds, dz).
\end{aligned}
\end{displaymath} Since $D$ is bounded, we have $\ex[M_{1}(t\wedge \tau)] = \ex[M_{2}(t\wedge \tau)] =0$ and, thanks to \eqref{eq-thm3.9-new-condtion}, we can show $\ex[M_{3}(t\wedge \tau)^{2}] < \infty$ and hence $\ex[M_{3}(t\wedge \tau)] =0$.

Now it follows from \eqref{bvp} that
\bea \ad \ex_{x,i} \left[u(X   (t\wedge \tau),\al
(t\wedge \tau)) \exp \( -\int^{t\wedge \tau}_0
 c(X   (s),\al(s))  ds\) \right] - u(x,i) \\
\aad \
= \ex_{x,i} \left[ \int^{t\wedge \tau}_0 \exp \( -\int^s_0
 c(X   (r),\al(r))  dr\)
 \left[\op  -c(X(s-),\al(s-)) \right] u(X   (s-),\al    (s-))ds \right]  \\
\aad \  = \ex_{x,i} \left[  \int^{t\wedge \tau}_0\exp \(- \int^s_0
c(X   (r),\al(r))  dr\)
\xi(X    (s-),\al   (s-))   ds \right].
\eea
Since $\xi\cd $ is bounded and $c\cd \ge 0$,  Lemma \ref{lem-3.8} and the bounded convergence theorem lead to
$$  \begin{aligned}\ex_{x,i} & \left[  \int^{t\wedge \tau}_0\exp \(- \int^s_0
c(X   (r),\al(r))  dr\)
\xi(X    (s-),\al   (s-))   ds \right] \\ &  \to  \ex_{x,i} \left[  \int^{ \tau}_0\exp \(- \int^s_0
c(X   (r),\al(r))  dr\)
\xi(X    (s-),\al   (s-))   ds \right] \end{aligned}$$ as $t\to \infty$.
On the other hand, using the continuity of the function $u$ and boundedness of the function $\eta\cd$, we have
% $$ \begin{aligned}
% & \ex_{x,i} \left[u(X   (t\wedge \tau),\al
% (t\wedge \tau)) \exp \( -\int^{t\wedge \tau}_0
% c(X   (s),\al(s))  ds\) \right]  \\
%  & \ \  \le \ex_{x,i}\left[u(X   (t\wedge \tau),\al
%(t\wedge \tau)) \right] \\
%& \ \  = \ex_{x,i}[u(X(t),\al(t))I_{\set{t<  \tau}}] +  \ex_{x,i}[u(X(\tau),\al(\tau))I_{\set{t \ge  \tau}}]  \\
%& \ \  \le K +  \ex_{x,i} [\eta(X(\tau),\al(\tau))I_{\set{t \ge  \tau}}] < \infty.
%\end{aligned}$$
$$ \begin{aligned}
   \abs{u(X   (t\wedge \tau),\al
(t\wedge \tau))  e^{  -\int^{t\wedge \tau}_0
 c(X   (s),\al(s))  ds} }
  & %\ \
   \le \abs{ u(X   (t\wedge \tau),\al
(t\wedge \tau))} \\
&  = \abs{u(X(t),\al(t))I_{\set{t<  \tau}}} +   \abs{u(X(\tau),\al(\tau))I_{\set{t \ge  \tau}}}  \\
&  \le K_{1} +   \abs{ \eta(X(\tau),\al(\tau))I_{\set{t \ge  \tau}}} \le K_{1 } + K_{2} < \infty,
\end{aligned}$$ where $K_{1}:= \max\{|u(x,i)|: x\in \lbar D, i \in \M\}$ and $ K_{2}:=\sup\{|\eta(x,i)|: x\in  D^{c}, i \in \M\}$.
Thus we can again apply the bounded convergence theorem and Lemma \ref{lem-3.8} to derive
$$ \begin{aligned}& \ex_{x,i} \left[u(X   (t\wedge \tau),\al
(t\wedge \tau)) \exp \( -\int^{t\wedge \tau}_0
 c(X   (s),\al(s))  ds\) \right] \\ & \ \ \to \ex_{x,i} \left[u(X   ( \tau),\al
( \tau)) \exp \( -\int^{ \tau}_0
 c(X   (s),\al(s))  ds\) \right] \ \text{ as } \ t\to \infty. \end{aligned}$$
Putting the above equations together leads to \eqref{fk}.
 \end{proof}

\section{Arcsine Laws}\label{sec:arc}

Paul L\'evy's celebrated arcsine law \cite{levy1940} states that for all $a\in (0,1)$, $$\pr\set{\int_{0}^{t} I_{\set{W(s) > 0}}ds \le a } = \frac{2}{\pi}\arcsin (\sqrt a). $$ The arcsine law plays a crucial role in the theory of fluctuations in
random walks.
In fact, it was shown in \cite{erdos1947} that for any  random walk $\{S_{n}\}_{n=1}^{\infty}$ with
zero mean and unit variance, $$ \lim_{n\to \infty} \pr\set{\frac{1}{n} \sum_{i=1}^{n}I_{\set{S_{i} >0}} \le a }= \frac{2}{\pi}\arcsin (\sqrt a).$$
Later, the arcsine laws were generalized in \cite{Khasm99,lamperti1958, watanabe1995}, to name just a few.

We aim to derive an arcsine law for regime-switching jump diffusion processes considered in this paper. To be precise, we consider a one-dimensional  regime-switching jump diffusion process, in which for simplicity, we assume that the drift term $b$ is identically zero. Moreover, motivated by applications in mathematical finance \cite{basu2009asymptotic}, risk modeling \cite{yin2006bounds}, queueing networks \cite{YinZ13},
etc., we suppose that the switching process is singularly perturbed with fast switching (see
 Section \ref{sect-weak_limit} for the precise formulation).  We show that as $\e \to 0$, the regime-switching jump diffusion process converges weakly and that  the limiting process has the arcsine law. To this purpose, we first show in  Section \ref{sect-weak_limit}  that under certain assumptions, the regime-switching jump diffusion process with fast switching converges weakly to a diffusion process. This result, together with the arcsine law for null recurrent diffusion established in \cite{Khasm99}, helps us to derive the desired arcsine law for regime-switching jump diffusion processes in Proposition \ref{lim-d-T}, which is in Section \ref{sect-arcsine-law}. Finally we remark in Section \ref{sect-L2-convergence} that  in general  there is no $L_{2}$ convergence associated with the weak convergence result established in Section   \ref{sect-weak_limit}.

\subsection{Weak Limit of Switching Jump Diffusions with Fast Switching}\label{sect-weak_limit}
For one-dimensional diffusions, it is well-known (see, e.g.,  \cite[Section 5.5]{Karatzas-S}) that through a proper transformation, one can remove  the drift term.
For simplicity,
we consider a one-dimensional regime-switching jump diffusion without drift   of the following form
\begin{equation}\label{sw-a} \barray
\ad dX^\e(t) = \sigma(X^\e(t),\al^\e(t))dW(t)  + \e \int_{\rr_0} \gamma(X^\e(t-),\al^\e(t-),y)\tilde N(dt,dy),
 \ t \ge   0, \\
 \ad  X^\e(0)=x_0, \ \al^\e(0)=i_0 \in \M ,
  \earray \end{equation}
where $\rr_0:= \rr-\set{0}$ (consistent with the definition $\rr^n_0$  with $n=1$) and
$\e>0$ is a small parameter.
We assume $x_0$  to be independent of
$\e$  and non-random for simplicity.
Throughout this section, we   assume that
the switching process is a continuous-time Markov chain $\al^\e(t)$ taking values
in $\M= \set{1, \dots, m}$. Moreover,  the Markov chain is
time-homogeneous with a generator $Q^\e= Q/\e$ such that $Q$ is weakly irreducible
\cite[p. 23]{YinZ13}. Denote the associated
quasi-stationary distribution by $\nu=(\nu_1,\dots,\nu_m) \in \rr^{1\times m}$.
As in Section \ref{sect-formulation}, we assume that  the Markov chain $\al^{\e} \cd$, the Brownian motion $W\cd$, and the Poisson random measure
$N\cd$ are independent.

Because of the assumption that $Q^{\e} = Q/\e$,  $\al^{e}\cd$ is the so-called singularly perturbed process with fast switching.
 The motivation for considering such  systems  stems from a wide range of applications in
flexible manufacturing systems, production planning, queueing networks, mathematical finance, risk modeling, etc.
The main goal is to obtain limiting systems with    reduced   complexity that still have good approximation properties to the original systems.
For example, in \cite{YinZ13},  starting with the motivation of a production planning problem in Section 1.1, Sections 3.3-3.5
present a number of examples of queues with finite capacity,
system reliability, competing risk  models, stochastic optimization problems, and linear quadratic control problems among others. It was demonstrated
 that a two-time-scale approach using $Q^{\e} = Q/\e$  leads to simpler limiting systems with good approximating properties.
%On the other hand,
Additionally,
using such a two-time-scale approach,
as in \cite{MW1998}, one can consider time-inhomogeneous Markovian models with {\em slowly varying} rates as well.
The reader is referred to the aforementioned references for further reading.

We proceed to obtain the  weak convergence theorem of this section. Let $X^{\e}\cd$ be the solution to \eqref{sw-a}. Then Theorem \ref{ep-lim}
indicates that as $\e\to 0$, $X^\e\cd$ converges weakly to $X\cd$, where $X\cd$ is a solution to the
following stochastic differential equation
\begin{equation}\label{lim-sw}  \begin{cases} dX(t) = \lbar \sigma (X(t)) dW(t), & \ t\ge 0,\\  X(0) =x_{0}, &  \end{cases}\end{equation}
with
\beq{sg-bar-def}  \lbar\sigma(x) =\sqrt{ \sum^m_{i=1} \sigma ^2(x, i) \nu_i}.\eeq

 \begin{rem}  Define $$ I(\lbar \sigma) : =\set{x\in \rr: \int_{-\delta}^{\delta} \frac{dy}{\lbar\sigma^{2}(x+y)}= + \infty, \forall \delta >0}, \quad Z(\lbar \sigma):=\set{x\in\rr: \lbar\sigma(x) =0}.$$
By virtue of  \cite[Theorems 5.5.4 and 5.5.7]{Karatzas-S},
\begin{itemize}
  \item the SDE \eqref{lim-sw} has a non-exploding weak solution for every initial distribution $\mu$ if and only if $I(\lbar \sigma) \subset Z(\lbar\sigma)$,
  \item  the SDE \eqref{lim-sw} has  a solution that is unique in the sense of probability law if and only if $I(\lbar \sigma) = Z(\lbar \sigma)$.
\end{itemize}
Suppose that $\sigma(\cdot, i)$ is continuous and $\nu_{i} > 0$ for each $i \in \M$. Then it is straightforward  to show that $I(\lbar \sigma) \subset Z(\lbar\sigma)$ and hence  \eqref{lim-sw} has a non-exploding weak solution. Suppose also that $Z(\lbar \sigma) = \emptyset$
(i.e., the diffusion is non-degenerate).
Then the inclusion $Z(\lbar \sigma)  \subset I(\lbar \sigma)$ is trivially satisfied and thus  \eqref{lim-sw}  has a unique weak solution.
\end{rem}

\begin{thm}\label{ep-lim}
Assume  condition
{\rm(A1)} and
suppose that $Q$ is weakly
irreducible and equation \eqref{lim-sw} has a unique $($in the sense of probability law$)$ solution for each
initial condition. Then for $0<T<\infty$, $\{X^\e\cd\}$ is tight in $D([0,T]: \rr)$ and any weakly
convergent subsequence has limit $X\cd$, which is the solution of \eqref{lim-sw}.
\end{thm}

\begin{proof} The proof is divided into
several steps.

Step 1:
We first prove the tightness of $\{X^\e\cd\}$  in $D([0,T]:\rr)$.
Note that by virtue of \lemref{lem-bdd-moments},
$\ex \sup_{t\in [0,T]} |X^\e(t)|^2 <\infty$.
For any $\delta>0$  and any $0<t,s < T$ with $s \le \delta$, we have
\begin{equation}\label{x2}\barray X^\e(t+s)- X^\e(t) \ad=\int^{t+s}_t \sg (X^\e(u),\al^\e(u)) dW(u)\\
\aad \
+ \e\int^{t+s}_t \int_{\rr_0} \ga (X^\e(u-),\al^\e(u-), y )\wdt N (du, dy).\earray\end{equation}
Use $\ex^\e_t$ to denote the conditioning on the $\sg$-algebra generated by
${\cal F}^\e_t = \{ X^\e(u): u\le t\}$. Then by the H\"older inequality and the Lipschitz continuity of $\sg(\cdot,i)$ for each $i$,
we have
\begin{equation}\label{2-mom-es}\barray
\ex^\e_t | X^\e(t+s)- X^\e(t)|^2 \ad \le   K  \int^{t+s}_t  [1+ \ex^\e_t|X^\e(u)|^2] du +  K e^\e(t+s,t), \earray\end{equation}
where
\begin{equation}\label{2-mom-es-1} e^\e(t+s,t)= \e^2 \ex^\e_t  \Big| \int^{t+s}_t \int_{\rr_0} \ga (X^\e(u-),\al^\e(u-), y) \wdt N ( du, dy) \Big|^2.\end{equation}
Using \lemref{lem-bdd-moments},
\eqref{2-mom-es-1} implies that
%\footnote{Only Assumption (A1) is needed here; (A1) not necessary.}
$$ \limsup_{\e\to 0} \ex e^\e(t+s,t) =0$$
It then follows that there is a $\wdt \mu^\e(\delta)$ such that
$$\ex^\e_t | X^\e(t+s)- X^\e(t)|^2  \le\ex^\e_t \wdt \mu^\e (\delta) \ \hbox{ for all } \ 0\le s \le \delta, \ t\le T,$$
and that
$$\lim_{\delta\to 0}
\limsup_{\e\to 0} \ex \wdt \mu^\e(\delta) =0.$$
Therefore, by the Kurtz tightness criterion \cite[Theorem 3, p.47]{Kushner84},
$\{X^\e\cd\}$ is tight in $ D([0,T],\rr)$.

Step 2: Characterize the limiting process. Since $\{X^\e\cd\}$ is tight, by Prohorov's theorem, we can extract
a weakly convergent subsequence.   With a slight abuse of notation, let's still index such a subsequence by 
 %Do so and still index the subsequence by
  $\{X^\e\cd\}$ with limit denoted by $X\cd$.
In view of the Skorohod representation theorem, we may assume % (with a slight abuse of notation), 
without loss of generality that $X^\e\cd$ converges to
$X\cd$ a.s., and the convergence is uniform on any bounded interval.  We proceed to characterize the limiting process
using a martingale problem formulation.
We pick out
any $\rho\cd\in C^3_0$ (class of $C^3$ functions with compact support),
which implies that $D^2_x \rho\cd$ is Lipschitz. For any $t,s>0$,
\begin{equation}\label{mg}\barray
\disp \rho(X^\e(t)) -\ad  \int^t_0  \Big\{ {1\over 2}\sg^2(X^\e(u),\al^\e(u))D^2_x\rho(X^\e(u))  \\
\aad \qquad  +\e\int_{\rr_{0}}[\rho(X^\e(u)+ \gamma(X^\e(u),\al^\e(u),z),\al^\e(u))-\rho(X^\e(u),
\al^\e(u))\\
\aad \qquad\qquad -
      D_x \rho(X^\e(u),\al^\e(u))\gamma(X^\e(u),\al^\e(u),z) ] \nu(dz)\Big\}du\  \hbox{ is a martingale.}\earray\end{equation}
      Moreover,
\begin{equation}\label{dyn-again}\barray
\ad \ex^\e_t \rho(X^\e(t+s)) -\rho(X^\e(t))\\
\aad \ =  \ex^\e_t\int^{t+s}_t\Big\{ {1\over 2}\sg^2(X^\e(u),\al^\e(u))D^2_x\rho(X^\e(u)) \\
\aad \qquad \ +   \e \int_{\rr_{0}}[\rho(X^\e(u)+ \gamma(X^\e(u),\al^\e(u),z),\al^\e(u))-\rho(X^\e(u),
\al^\e(u))\\
\aad \qquad \hfill -
      D_x \rho(X^\e(u),\al^\e(u))\gamma(X^\e(u),\al^\e(u),z) ] \nu(dz)\Big\}du  \ \hbox{ a.s.}\earray\end{equation}

Choose an arbitrary real-valued, bounded and continuous function $\ph\cd$,   positive integer $\kappa$,  $t>0$, $s>0$,
and $t_l \le t$ with $l \le \kappa$, we have
\begin{equation}\label{mg-1}\barray \ad \ex \ph(X^\e(t_l): l\le \kappa)[ \rho(X^\e(t+s))- \rho(X^\e(t))]\\
\aad \ \to    \ex \ph(X(t_l): l\le \kappa)[ \rho(X(t+s))- \rho(X(t))] \ \hbox{ as } \e\to 0\earray\end{equation}
by the weak convergence and the Skorohod representation, together with the continuity of $\ph\cd$ and $\rho\cd$.
Next, in view of the moment bounds of the jump term, the weak convergence, and the Skorohod representation, we also have
\begin{equation}\label{mg-2}\barray
\ad \ex \ph(X^\e(t_l): l\le \kappa)\Big\{\int^{t+s}_t \e \int_{\rr_{0}}[\rho(X^\e(u)+ \gamma(X^\e(u),\al^\e(u),z),\al^\e(u))-\rho(X^\e(u),
\al^\e(u))\\
\aad \qquad \hfill -
      D_x \rho(X^\e(u),\al^\e(u))\gamma(X^\e(u),\al^\e(u),z) ] \nu(dz)du\Big\}\\
\aad \ \to 0 \ \hbox{ as } \ \e \to 0.\earray\end{equation}

Dividing $[t,t+s]$ into subintervals of length $\Delta_\e= \e^{(1-\Delta)/2}$ for some $0<\Dl<1$.
We have
\begin{equation}\label{div}\barray\ad \int^{t+s}_t \sg^2(X^\e(u),\al^\e(u)) D^2_x \rho(X^\e(u)) du \\
\ad \ = \sumd \!\intl\! \sg^2(X^\e(t+l\Dl^\e),\al^\e(u)) D^2_x \rho(X^\e(t+l\Dl^\e)) du
 % \\ \aad \ \hfill
  +  e^\e_1(t+s,t), % + e^\e_2(t+s,t),
\earray\end{equation}
where
\begin{equation}\label{mg-2a}\barray \ad\!\! e^\e_1(t+s,t)= \!\sumd \!\intl\!\! \Big[\sg^2(X^\e(t+l\Dl^\e),\al^\e(u)) [D^2_x \rho(X^\e(u))- D^2_x \rho(X^\e(t+l\Dl^\e))]  \\ \aad \hfill  + [\sigma^2(X^\e(u),\al^\e(u))- \sigma^2(X^\e(t+ l \Delta^\e),\al^\e(u))] D^2_x \rho(X^\e(u)) \Big]du.
\earray\end{equation}
Using the boundedness and smoothness of $\rho\cd$, the Lipschitz continuity of $\sg(\cdot, \al)$,
and \eqref{mg-2a}, it is easily seen that
\begin{equation}\label{mg-3}
\ex \ph(X^\e(t_l): l\le \kappa)  e^\e_1(t+s,t)
 \to 0 \ \hbox{ as } \ \e\to 0.\end{equation}
Concentrating on the term in the second line of \eqref{div}, we have
\bea \ad \ex \ph(X^\e(t_l): l\le \kappa)\Big[
\sumd \intl \sg^2(X^\e(t+l\Dl^\e),\al^\e(u)) D^2_x \rho(X^\e(t+l\Dl^\e)) du\Big]\\
\aad \ = \ex \ph(X^\e(t_l): l\le \kappa)
\Big[\sumd \ett\intl  \sg^2(X^\e(t+l\Dl^\e),\al^\e(u))\\
\aad \ \hfill \times D^2_x \rho(X^\e(t+l\Dl^\e)) du\Big]\\
\aad \ =  \ex \ph(X^\e(t_l): l\le \kappa)
\Big[\sumd \sum_{i=1}^m \sg^2(X^\e(t+l\Dl^\e),i) D^2_x \rho(X^\e(t+l\Dl^\e))\\
\aad \quad \hfill \times \ett\intl [I_{\{\al^\e(u)=i\}} - \nu_i] du \Big]\\
\aad \ \hfill +  \ex \ph(X^\e(t_l): l\le \kappa)
\Big[\sumd \sum_{i=1}^m \sg^2(X^\e(t+l\Dl^\e),i) D^2_x \rho(X^\e(t+l\Dl^\e)) \nu_i \Dl^\e\Big].
\eea
Using similar estimates as in \cite[Lemma 5.35 (a)]{YinZ13},
\bea \disp \l\ett\intl [I_{\{\al^\e(u)=i\}} - \nu_i] du \r
\ad   \le K \Big[\ett \Big| \intl [I_{\{\al^\e(u)=i\}} - \nu_i] du\Big|^2\Big]^{1/2}\\
\ad  \le  O(\sqrt \e).\eea
By the boundedness of $D^2_x \rho\cd$, the linear growth condition of $\sg(\cdot, i)$, and
the estimate of the second moment $\ex\sup_{0\le t\le T}|X^\e(t)|^2<\infty,$ we have
\begin{equation}\label{mg-4}\barray
\ad \ex \l\sumd \sg^2(X^\e(t+l\Dl^\e),i) D^2_x \rho(X^\e(t+l\Dl^\e)) \ett\intl [I_{\{\al^\e(u)=i\}} - \nu_i] du \r \\
\aad \le K {1\over \Dl^\e} \sqrt \e \le K \e^{\Dl/2} \to 0 \ \hbox{ as } \ \e\to 0.\earray\end{equation}
Next, we obtain
\begin{equation}\label{mg-5}\barray\ad
\ex \ph(X^\e(t_l): l\le \kappa)
\Big[\sumd \sum_{i=1}^m \sg^2(X^\e(t+l\Dl^\e),i) D^2_x \rho(X^\e(t+l\Dl^\e)) \nu_i \Dl^\e\Big]\\
\aad \ \to \ex\ph(X(t_l): l\le \kappa)
\Big[\int^{t+s}_t \sum_{i=1}^m  \sg^2(X(u),i) D^2_x \rho(X(u)) \nu_i du \Big] \ \hbox{ as } \ \e\to 0.\earray\end{equation}

Combining \eqref{mg-1}--\eqref{mg-5}, we obtain that
\bea \ad
\ex \ph(X(t_l): l\le \kappa)\Big[\rho(X(t+s)-\rho(X(t))- \int^{t+s}_t {1\over 2}\sum_{i=1}^m  \sg^2(X(u),i)\nu_i D^2_x \rho(X(u))  du\Big]=0.\eea
That is,
$$\rho(X(t))- \int^t_0 {1\over 2} \lbar \sg^2(X(u)) D^2_x \rho(X(u)) du \ \hbox{ is a martingale.}$$
Equivalently, $X\cd$ is a solution of the martingale problem with operator
$$L \rho(x)= {1\over 2} \lbar \sg(x) D^2_x\rho(x).$$ The desired result thus follows.
\end{proof}

\subsection{Arcsine Laws}\label{sect-arcsine-law}
The   arcsine law due to  \cite{Khasm99} is concerned with null recurrent diffusions.
Such null recurrent Markov processes were studied in details in \cite{K80} and
refined and more verifiable conditions were given in \cite{KY00}.
Note that a necessary and sufficient condition for the one-dimensional diffusion process
$X(t)$ given in \eqref{lim-sw} to be null recurrent is
$$ 2\int^\infty_{-\infty}p(x) dx = \pm\infty, \text{ where } p(x)= \frac{1}{\lbar \sigma^2(x)}.$$

\begin{itemize}
\item[(A3)] The following conditions hold:
\begin{itemize}

\item[$\bullet$] $\lbar \sg^2(x) \ge \sg_0>0$ for all $x\in \rr$,
 and for some $p_{\pm} \in \rr$,
\beq{ave-p} \lim_{L\to \pm \infty} {1\over L} \int^L_0 p(x) dx  = p_{\pm};\eeq
\item[$\bullet$]
for a bounded and piecewise continuous function $f: \rr \mapsto \rr$, there exist
$f_+ \not = f_-$ such that
\beq{ave-f} \lim_{L \to \pm \infty}{\disp\int^L_0 f(x) p(x) dx \over \disp \int^L_0 p(x) dx} = f_{\pm}.\eeq
\end{itemize}
\end{itemize}
Now let us state the  arcsine law given in \cite{Khasm99}.

\begin{lem}\label{arc}
Assume {\rm (A3)}. Consider \eqref{lim-sw} with $\lbar \sg(x)$ given in \eqref{sg-bar-def} and
define
$$ \eta(T)= {\disp {1\over T} \int^T_0 f(X(t)) dt -f_- \over \disp  f_+-f_-} .$$
Then the following results hold.
\begin{itemize}
\item[{\rm(i)}] When $p_+=p_-$, the limiting distribution is the arcsine law
\beq{arc-law} \lim_{T\to\infty} \pr (\eta(T) < z)={2\over \pi} \arcsin \sqrt z, \  z \in [0,1].\eeq
\item[{\rm(ii)}] When $p_+\not = p_-$, the limiting distribution coincides with the distribution of
a random variable $\Xi$ such that for all $z>0$ and $A= \sqrt {p_+/p_-}$,
\beq{law-1} \ex {1\over z +\Xi} ={\sqrt {1+z} + A \sqrt z \over \sqrt { (1+z)z} (\sqrt z + A \sqrt { 1+z})}.\eeq
The distribution of $\Xi$ is uniquely determined by \eqref{law-1}.
\end{itemize}
\end{lem}

\lemref{arc} presents an arcsine law for the limiting diffusion process.
We further obtain a limiting distributional result.
By virtue of  \thmref{ep-lim}, $X^\e\cd$ converges weakly to $X\cd$
in $D[0,\infty)$ (that is,
 for any $T <\infty$,
 $X^\e\cd$ converges weakly to $X\cd$
in $D[0,T]$).
Furthermore, using perturbed Lyapunov function argument (see \cite[Chapter 4]{Kushner84}), it is not difficult to
demonstrate that $\{X^\e(t): t>0, \ \e>0\}$ is tight.
Define
\bea \ad  \eta(t)= {\disp {1\over t} \int^t_0 f(X(s)) dt -f_- \over \disp  f_+-f_-}, \ \ t \in [ 0,T],\eea
and
\bea \ad  \eta^\e(t)= {\disp {1\over t} \int^t_0 f(X^\e(s)) dt -f_- \over \disp  f_+-f_-},\  \ t \in [ 0,T].\eea
Since $f\cd$ is a bounded and piecewise continuous function,
\thmref{ep-lim} and \cite[Corollary 5.2, p.31]{Billingsley-68} yield that $f(X^\e\cd)$
converges weakly to $f(X\cd)$ in $D[0,T]$ for any $T<\infty$. Thus,
$\eta^\e\cd$ converges weakly to $\eta\cd$.
By \lemref{arc} together with the tightness of $\{X^\e(t):$ $t>0$,  $\e>0\}$,
$\eta(T)$ as  random variables converge  in distribution to a random variable $\wdh \eta$ such that either
(i) or (ii) in \lemref{arc} holds.
We summarize this into the following result.

\begin{prop}\label{lim-d-T} Under the conditions of \thmref{ep-lim} and \lemref{arc},
\begin{itemize}
\item[{\rm(i)}] when $p_+=p_-$,
\beq{-eq2-arc-law} \lim_{T\to\infty} \lim_{\e\to 0} \pr (\eta^\e(T) < z)={2\over \pi} \arcsin \sqrt z, \  z \in [0,1],\eeq
\item[{\rm(ii)}] when $p_+\not = p_-$, the limiting distribution
of $\eta^\e(T)$ $($as $\e\to 0$ and then ${T\to\infty})$ is the same as  the distribution of
a random variable $\Xi$ given in \eqref{law-1} for all $z>0$ with $A= \sqrt {p_+/p_-}$.
%\beq{law-1a} \ex {1\over z +\Xi} ={\sqrt {1+z} + A \sqrt z \over \sqrt { (1+z)z} (\sqrt z + A \sqrt { 1+z})}.\eeq
The distribution of $\Xi$ is uniquely determined by \eqref{law-1}.
\end{itemize}
\end{prop}

\subsection{There Is No $L_2$ Convergence}\label{sect-L2-convergence}
We have established weak convergence in \thmref{ep-lim}. One natural question is: Can we get a stronger convergence in the sense of
$L_2$? The following example gives a negative answer to the question.

\begin{exm}\label{example1}
Consider a regime-switching diffusion
\begin{equation}
\label{ex-sw-diffusion}
dX^\e(t)= \sigma(\alpha^\e(t)) dW(t), \ \ X^\e(0)=x \in \rr,
\end{equation}
where $\sigma(1)= \sigma_1 \not= \sigma(2)= \sigma_2 $, and $\alpha^\e\in \set{1,2}$ is a continuous-time Markov chain  generated by $Q/\e$ with $Q=\begin{pmatrix}-q_1 & q_1\\ q_2 & -q_2 \end{pmatrix}$ and $q_1, q_2 >0$.   By virtue of Theorem \ref{ep-lim}, $X^\e$ converges weakly to $X$, the solution to
$$dX(t)=  \lbar \sigma dW(t), \ \ X(0)= x\in \rr,$$
where $\lbar \sigma = \sqrt{\sg_1^2 \nu_1 + \sg^2_2 \nu_2}$, and $\nu_1 = \frac{q_2}{q_1+ q_2} $, $\nu_2 = \frac{q_1}{q_1+ q_2} $. Note that
\begin{displaymath}
\ex_{x,1}  \abs{X^\e(t) - X(t)}^2 = \ex_{x,1} \abs{\int_0^t \left[ \sigma(\al^\e(u)) -\lbar \sigma \right]  dW(u)}^2  =
 \ex_{x,1}\int_0^t \abs{\sigma(\al^\e(u)) -\lbar \sigma}^2 du.
\end{displaymath}
Using the Kolmogorov forward equation, one can find that
\begin{displaymath}
\pr_1\set{\al^\e(u) =1} = 1- \pr_1\set{\al^\e(u) =2} = \nu_1 + \nu_2 e^{-(q_1+ q_2) u /\e}.
\end{displaymath}
Then it follows that
\begin{displaymath}
\begin{aligned}
\ex_1  \abs{\sigma(\al^\e(u)) -\lbar \sigma}^2  & = \abs{\sigma_1 -\lbar \sigma}^2\pr_1 \set{\al^\e(u) =1} + \abs{\sigma_2 -\lbar \sigma}^2\pr_1 \set{\al^\e(u) =2}  \\
& = \abs{\sigma_1 -\lbar \sigma}^2 (\nu_1 + \nu_2  e^{-(q_1+ q_2) u /\e}) + \abs{\sigma_2 -\lbar \sigma}^2(\nu_2 -\nu_2 e^{-(q_1+ q_2) u /\e})\\
 & = 2\lbar \sigma [\lbar \sigma -(\sigma_1\nu_1+ \sigma_2\nu_2)] - (\sigma_1  -\sigma_2)(\sigma_1 + \sigma_2 + 2\lbar \sigma) \nu_2 e^{-(q_1+ q_2) u /\e}.
\end{aligned}
\end{displaymath}
Consequently we have
\begin{displaymath}
\begin{aligned}
\lefteqn{ \ex_{x,1}   \abs{X^\e(t) - X(t)}^2 } \\ &\ \  = 2\lbar \sigma [\lbar \sigma -(\sigma_1\nu_1+ \sigma_2\nu_2)]  t - (\sigma_1  -\sigma_2)(\sigma_1 + \sigma_2 + 2\lbar \sigma)(1- e^{-(q_1+ q_2) t /\e}) \frac{\nu_2 \e}{q_1+ q_2}  \\
&\ \  \to  2\lbar \sigma [\lbar \sigma -(\sigma_1\nu_1+ \sigma_2\nu_2)]  t, \text{ as } \e\to 0.
\end{aligned}\end{displaymath}
Note that  under the condition $\sigma_1 \not =\sigma_2$, one can immediately  verify that $\lbar \sigma -(\sigma_1\nu_1+ \sigma_2\nu_2) > 0$ and hence $ \ex_{x,1}   \abs{X^\e(t) - X(t)}^2 \not\to 0$ as $\e\to 0$.
\end{exm}

\section{Further Remarks}\label{sec:conclu}
This paper has been devoted to revealing the connections of regime-switching jump diffusions with a class of coupled
systems of partial integro-differential    equations.
Under broad conditions, we have obtained serval versions of the Feynman-Kac formulas together with
the associated initial, terminal, and boundary value problems.
Moreover, certain limiting results have been
obtained for processes with fast switching. In addition,  arcsine laws for a limiting process enables us
to draw conclusion for certain systems with two-time-scale formulation.

In this work, the jump part is driven  by  a Poisson random measure associated with a L\'evy process.
 A worthwhile future effort is to treat systems in which
the random driving force is  an alpha-stable process that has finite
$p$th moment with $p<2$. This requires more work and careful consideration.
In this paper, when the   switching component $\al$ switches at time $\tau$ or $\al(\tau) \not= \al(\tau-)$, it is assumed
the $X$ component
is fixed or $X(\tau)= X(\tau-)$.  A relevant question is: Can we allow the $X$ component to take place from $x$ in one plane to $y$ in another
plane at the instant of a switching?
 Mathematically, the switching part will also be represented by an integral operator as in the formulation of \cite{il1999asymptotic}. This adds another fold of
difficulty.  With the aid of the Feynman-Kac formulas,  we may proceed to treat many stochastic control problems.
For example,
combining real options, game theory, and a regime-switching formulation with jumps, we may consider
an irreversible investment problem with Stackelberg leader-follower competition and  market regime changes (see the related work
using switching diffusion formulation \cite{BeHYY14}).  The treatment of the real options and the related problems with competition
have received resurgent attention lately.
Furthermore, effort may also be directed to applications and extensions to ratchet theory for molecular motor,
 and stochastic dynamics of electrical membrane with voltage-dependent ion channel fluctuations.
All of these will be worthwhile future efforts.

\def\cprime{$'$}

\end{document}